\documentclass[a4paper]{amsart}

\usepackage{amsmath,amsrefs,amssymb,ifthen}
\usepackage{mathtools}
\usepackage{hyperref}
\usepackage{enumerate}

\usepackage[a4paper,top=2.5cm,bottom=2.5cm,left=2.5cm,right=2.5cm,marginparwidth=1.25cm]{geometry}

\newtheorem{thmintro}{Theorem}

\newtheorem{corintro}[thmintro]{Corollary}

\newtheorem{thm}{Theorem}[section]
\newtheorem{prop}[thm]{Proposition}
\newtheorem{lem}[thm]{Lemma}
\newtheorem{cor}[thm]{Corollary}

\newtheorem{question}[thm]{Question}

\theoremstyle{remark}
\newtheorem{rem}[thm]{Remark}
\newtheorem{example}[thm]{Example}

\newtheorem{notation}[thm]{Notation}

\theoremstyle{definition}

\newtheorem{defi}[thm]{Definition}

\newcommand{\Z}{\mathbb{Z}}

\newcommand{\R}{\mathbb{R}}
\newcommand{\N}{\mathbb{N}}

\DeclarePairedDelimiter{\abs}{\lvert}{\rvert}
\DeclarePairedDelimiter{\Span}{\langle}{\rangle}

\newcommand{\Axis}[1]{\operatorname{Axis}(#1)}
\newcommand{\actson}{\curvearrowright}

\newcommand{\boldparagraph}[1]{\paragraph{\textbf{#1}}}

\newcommand{\cat}[1]{\ensuremath{\operatorname{CAT}(#1)}}
\newcommand{\cay}[1]{\operatorname{Cay}(#1)}

\newcommand{\Edge}[1]{\operatorname{Edge}(#1)}
\renewcommand{\emptyset}{\varnothing}
\renewcommand{\epsilon}{\varepsilon}

\renewcommand{\phi}{\varphi}
\newcommand{\fix}[1]{\operatorname{Fix}(#1)}

\newcommand{\goodname}{stabilisation}

\newcommand{\Ker}[1]{\operatorname{Ker}(#1)}

\newcommand{\Min}[1]{\operatorname{Min}(#1)}

\newcommand{\norm}[2]{\operatorname{N}_{#1}(#2)}

\newcommand{\pstab}[2]{\operatorname{Stab}^{\operatorname{pt}}_{#1}(#2)}

\newcommand{\QZ}[2]{\operatorname{QZ}_{#1}(#2)}

\renewcommand{\setminus}{\smallsetminus}
\newcommand{\stab}[2]{\operatorname{Stab}_{#1}(#2)}
\newcommand{\stabfix}[1]{\operatorname{Fix}^{\infty}(#1)}

\newcommand{\utl}[1]{\operatorname{utl}(#1)}

\newcommand{\Ver}[1]{\operatorname{Vert}(#1)}

\usepackage{xcolor}

\newcommand{\diam}{\mathrm{diam}}

\newcommand{\integers}{\mathbb Z}
\newcommand{\paclass}{\mathcal{PA}}
\newcommand{\upaclass}{\mathcal{UPA}}
\newcommand{\reals}{\mathbb R}
\newcommand{\basol}{\operatorname{BS}}
\newcommand{\euclidean}{\mathbb E}

\newcommand{\showcommentsbox}{yes}

\newsavebox{\commentbox}
{\ifthenelse{\equal{\showcommentsbox}{yes}}%
{\footnotemark
        \begin{lrbox}{\commentbox}
        \begin{minipage}[t]{1.25cm}\raggedright\sffamily\tiny
        \footnotemark[\arabic{footnote}]}
{\begin{lrbox}{\commentbox}}}%
{\ifthenelse{\equal{\showcommentsbox}{yes}}%
{\end{minipage}\end{lrbox}\marginpar{\usebox{\commentbox}}}
{\end{lrbox}}}

\newcommand{\paf}{\mathrm{PA}}

\begin{document}


\title{Combination theorems  for  Wise's power alternative}

\author[M. Hagen]{Mark Hagen}
\address{(Mark Hagen) School of Mathematics, University of Bristol, Bristol, UK}
\email{markfhagen@posteo.net}

\author[A. Martin]{Alexandre Martin}
\address{(Alexandre Martin) Department of Mathematics, Heriot-Watt University and Max\-well Institute for Mathematical Sciences, Edinburgh, UK}
\email{alexandre.martin@hw.ac.uk}

\author[G. Sartori]{Giovanni Sartori}
   \address{(Giovanni Sartori) Department of Mathematics, Heriot-Watt University and Max\-well Institute for Mathematical Sciences, Edinburgh, UK}
   \email{gs2057@hw.ac.uk}

\begin{abstract} We show that Wise's power alternative is stable under certain group constructions, use this to prove the power alternative for new classes of groups, and recover known results from a unified perspective.

For groups acting on trees, we introduce a dynamical condition that allows us to deduce the power alternative for the group from the power alternative for its stabilisers of points. As an application, we reduce the power alternative for Artin groups to the power alternative for free-of-infinity Artin groups, under some conditions on their parabolic subgroups. 
We also introduce a uniform version of the power alternative and prove it, among other things, for a large family of two-dimensional Artin groups. As a corollary, we deduce that these Artin groups have uniform exponential growth. 

Finally, we prove that the power alternative is stable under taking relatively hyperbolic groups. We apply this to show that various examples, including all free-by-$\mathbb{Z}$ groups and a natural subclass of hierarchically hyperbolic groups, satisfy the uniform power alternative.
\end{abstract}

\maketitle

\small



\normalsize

\tableofcontents

\section{Introduction}

\subsection{Motivation} We say that a group~$G$ satisfies the \emph{power alternative} (or \emph{Wise's power alternative} in~\cite{martin2023tits}) if, for every~$g,h\in G$, there exists~$n\in\N_{\ge1}$ such that either~$g^n$ and~$h^n$ commute or they generate non-abelian free group. This alternative is closely related to the Tits alternative for non-positively curved groups, and Wise asked whether all \cat{0} groups (or non-positively curved groups  in a broader sense) satisfy it~\cite{bestvina2004questions}*{Question~2.7}. While Leary and Minasyan recently provided the first example of a \cat{0} group not satisfying the power alternative ~\cite{leary2021commensurating}*{Example~9.4}, this alternative has been established for many non-positively curved groups by various means.  For example, in 1979, Jaco-Shalen proved a strong form of the power alternative for $G=\pi_1M$ where $M$ is an atoroidal Haken $3$--manifold \cite{JacoShalen}*{Theorem VI.4.1}.  If $G$ is a hyperbolic group, standard arguments provide some $n\in\N_{\ge1}$ such that for all $g,h\in G$, either $\Span{g^n,h^n}_G\cong\Z$ or $\Span{g^n,h^n}_G\cong F_2$~\cite{loh2017geometric}*{Theorem~8.3.13}.  
For~$G$ a right-angled Artin group, Baudisch proved that the power alternative holds for $G$ with $n=1$, that is, two elements of a right-angled Artin group either commute or generate a non-abelian free subgroup~\cite{baudisch1981subgroups}. In particular (see Lemma \ref{lem:wpa virtual property}) the power alternative holds for \emph{virtually special} groups in the sense of~\cite{HaglundWise:special}, which implies the power alternative for Coxeter groups, among many other examples (see Section \ref{sec:recovering}). The power alternative for right-angled Artin groups also implies that same alternative for mapping class groups, by work of Koberda~\cite{koberda2012RAAGinMCG}*{Corollary~1.2}. 
Very recently, the power alternative has been established for mapping tori of many injective free group endomorphisms, and thus many free-by-$\Z$ groups \cite{ABKSV}.

Artin groups form a large class of groups generalising braid groups, and which are conjectured to be non-positively curved in some sense.  It is thus natural to ask:

\begin{question}[\cite{martin2023tits}*{Question~1.2}]
    Which Artin groups satisfy the power alternative?
\end{question}

Recently, Antol\'{i}n and Foniqi proved the power alternative for even Artin groups of FC-type~\cite{antolin2023subgroups}*{Theorem~1.1}, and Martin proved it for two-dimensional Artin groups of hyperbolic type~\cite{martin2023tits}*{Theorem~B}. In this article, we prove the power alternative for new classes of Artin groups.

\medskip 

In a different direction, it is natural to ask whether the power alternative is stable under various group constructions. Antol\'in--Minasyan proved that the power alternative is stable under taking graph products of groups~\cite[Corollary 1.5]{antolin2015tits}. In this article, we obtain several new combination theorems  for the power alternative. That is, given a group action on a simplicial complex, we provide conditions on the action so that the power alternative for stabilisers of simplices implies the power alternative for the whole group. We investigate in particular groups that arise from fundamental groups of graphs of groups (via their action on the associated Bass--Serre tree) and relatively hyperbolic groups (via their action on their coned-off Cayley graph). 
We thereby obtain a collection of combination theorems that recover several of the above results in a unified way, and allow us to obtain new ones.

\subsection{Splittings and applications to Artin groups}\label{subsec:intro-split-artin}

For a group acting on a tree, we first investigate when it is possible to derive the power alternative for the group from the power alternative for the stabilisers of points. Note that this does not hold in full generality: For instance, the Baumslag--Solitar group $\operatorname{BS}(1,2)$ can be written as an HNN extension of the form $\mathbb{Z}*_{\mathbb{Z}}$, and in particular acts on a tree with infinite cyclic point stabilisers, yet it does not satisfy the power alternative. In a different direction, the power alternative is still open for groups acting geometrically on a product of two trees, due to the presence of \emph{antitori}~\cite{Wise:CSC}. These groups decompose as fundamental groups of graphs of free groups, and as such act on simplicial trees with free groups as the cell stabilisers (see Example \ref{exmpl:irreducible-lattice}).

We introduce a \emph{\goodname} property (see Definition~\ref{def:goodname property}), which is a condition on the dynamics of the action that allows us to derive the power alternative for the whole group. This condition is in particular satisfied by acylindrical actions, see Example~\ref{ex:acylindrical}.

Our main criterion is the following. We phrase it for actions on simplicial trees as our applications deal with such actions, but the theorem holds more generally for actions on real trees:

\begin{thmintro}[power alternative for actions on trees]
    \label{thmintro:WPA for actions on trees}
    Let~$T$ be a simplicial tree and let $G$ be a group acting on~$T$ via graph automorphisms. Let us further assume that:
    \begin{enumerate}
        \item for every $x\in \Ver{T}\cup\partial T$, the power alternative holds for $\stab G x$;
        \item the action $G\actson T$ has the \emph{\goodname} property (see Definition~\ref{def:goodname property}). 
    \end{enumerate}
    Then the power alternative holds for $G$.
\end{thmintro}

We also obtain a more algebraic version of this criterion, where the condition on the stabilisers of points at infinity is replaced with a condition on stabilisers of points and certain associated ``mapping tori'', see Corollary~\ref{cor:law-with-mapping-tori}.

As an application, we first study Artin groups. It is known that if the presentation graph of an Artin group is not a complete graph, then the group admits a \emph{visual splitting} as an amalgamated product of standard parabolic subgroups (see Section~\ref{sec:background on artin groups}). 
We obtain the following general result:

\begin{corintro}
\label{corintro:wpa visual splitting artin}
     Let $A_\Gamma=A_{\Gamma_1}*_{A_{\Gamma_0}}A_{\Gamma_2}$ be an Artin group expressed as a visual splitting with the following properties: 
    \begin{enumerate}
         \item the power alternative holds for the factors~$A_{\Gamma_1}$ and~$A_{\Gamma_2}$;
         \item $A_{\Gamma}$ has the \emph{parabolic intersection property} and the \emph{normaliser structure property} (see Definition~\ref{def:intersection of parabolics, structure of normalisers}).
    \end{enumerate}  
    Then the power alternative holds for~$A_{\Gamma}$.
\end{corintro}

The parabolic intersection property and the normaliser structure property are conjectures about parabolic subgroups of Artin groups that are believed to hold for all Artin groups. So far, they have been proved only for certain sub-classes (see Section~\ref{sec:background on artin groups} for more details). Under the assumption that these two conjectures hold, Corollary~\ref{corintro:wpa visual splitting artin} reduces the power alternative for all Artin groups to the power alternative for Artin groups over complete graphs (the so-called \emph{free-of-infinity} Artin groups), see Corollary~\ref{cor:reduction_free_of_infinity}.

\medskip

\subsection{The uniform power alternative} In this article, we introduce a natural quantitative strengthening of the power alternative: We say that a group $G$ satisfies \emph{the uniform power alternative} if there exists a uniform integer $N\in\N_{\ge1}$ such that, for every $g,h\in G$, either $[g^N,h^N]=1$ or $\Span{g^N,h^N}_{G}\cong F_2$. This is  strictly stronger than the power alternative (see Example~\ref{ex:group with wpa but not uwpa}), and  holds for hyperbolic groups, right-angled Artin groups, and virtually special groups, for instance.
\par
Although it is straightforward to produce examples of groups that satisfy the power alternative but do not satisfy the uniform one, we do not know whether the two properties are equivalent for finitely-presented groups, whence the question:

\begin{question}
    Is there a finitely-presented (torsion-free) group that satisfies the power alternative but does not satisfy the uniform power alternative?
\end{question}

Dihedral Artin groups are known to satisfy the power alternative with a uniform exponent that only depends on the associated label~\cite{huang2016cocompactly}*{Lemma~4.3}. We exploit this result to verify the uniform power alternative for a wide class of two-dimensional Artin groups, and with a uniform exponent that depends only on the set of labels of the presentation graph:

\begin{thmintro}
    \label{corintro:uniform wpa}
    Let $A_\Gamma$ be an Artin group whose defining graph is $(2,2)$-free and triangle-free. For each edge $\{a,b\}$ of~$\Gamma$ with label $m_{ab}$, let
    \[
    m_{ab}'=
    \begin{cases}
        m_{ab}/2    &\text{if $m_{ab}$ is even}\\
        2m_{ab}      &\text{if $m_{ab}$ is odd}
    \end{cases}
    \]
    and let $N=\operatorname{lcm}\left\{m'_{ab}:\{a,b\}\in\Edge\Gamma\right\}$.
    The following holds:
    \begin{itemize}
        \item if $N \geq 3$, then $A_\Gamma$ satisfies the uniform power alternative with exponent $N$.
        \item if $N= 2$, then $A_\Gamma$ satisfies the uniform power alternative with  exponent $4$.
    \end{itemize}
\end{thmintro}

The uniform exponent defined in Theorem~\ref{corintro:uniform wpa} is optimal in the first case. In the second case, the expected optimal uniform exponent would be $N = 2$, however our approach yields a strictly larger constant in that case (see Lemma~\ref{cor:Uniform_hyp_hyp}, whose proof requires a constant at least $3$).

If all the labels are even, then an Artin group $A_\Gamma$ as in Theorem~\ref{corintro:uniform wpa} is an even FC type Artin group. The above theorem is thus related to Antolin--Foniqi's ``strongest Tits Alternative'' for these groups \cite{antolin2023subgroups}*{Theorem~1.1}. Indeed, one of the consequences of their main result is that such an Artin group $A_\Gamma$ satisfies a uniform power alternative with exponent 
    $N \coloneqq \prod_{\Edge\Gamma} m_{ab}'$. The exponent provided in Theorem~\ref{corintro:uniform wpa} is sharper (optimal, in most cases) and independent of the size of $\Gamma$.

As a corollary, we prove that the Artin groups considered in Theorem~\ref{corintro:uniform wpa} have a uniform exponential growth, in the following strong sense: 

\begin{corintro}\label{cor:UEG_Artin}
    Let $A_\Gamma$ be a triangle-free $(2,2)$-free Artin group that is not of spherical type. Let $m$ be a uniform exponent for the power alternative for $A_\Gamma$ (as defined for instance in Theorem~\ref{corintro:uniform wpa}).
    Let $S$ be a generating set of $A_\Gamma$. Then there exist $s, t \in S$ such that $s^m$ and $t^m$ generate a non-abelian free subgroup of $A_\Gamma$.
    In particular, $A_\Gamma$ has uniform exponential growth.
\end{corintro}

\subsection{Relative hyperbolicity}\label{subsec:relative-hyperbolicity}
In Section \ref{sec:recovering}, we consider the (uniform) power alternative for relatively hyperbolic groups, deducing the power alternative for the whole group from the power alternative for the peripheral subgroups. Our main result is the following:

\begin{thmintro}[Theorem~\ref{thm:rel-hyp}]\label{thmintro:rel-hyp}
Let $G$ be a finitely generated  group that is hyperbolic relative to a finite collection $\mathcal P$ of subgroups. 
Then $G$ satisfies the power alternative provided each $H\in\mathcal P$ does.  Moreover, if each $H\in\mathcal P$ has bounded torsion and satisfies the uniform power alternative, then $G$ satisfies the uniform power alternative.
\end{thmintro}

This generalises the well-known fact that hyperbolic groups satisfy the uniform power alternative (see Theorem \ref{thm:hyperbolic}).  Our proof is phrased as much as possible in terms of acylindricity of the action of $G$ on the coned-off Cayley graph, and 
we imagine that there might be a larger class of acylindrically hyperbolic groups for which one can deduce the (uniform) power alternative from the assumption that a sufficiently rich class of elliptic subgroups satisfy the power alternative uniformly.  In Section \ref{subsec:questions-HHG}, we ask some questions about an analogue of Theorem \ref{thmintro:rel-hyp} for various hierarchically hyperbolic groups.  In this direction, Definition \ref{defn:WPA-class} specifies classes of groups shown in Theorem \ref{thm:meta-pa} to satisfy the power alternative (uniformly).  These classes are constructed from some initial groups by closing the class under various operations: passing to finite-index subgroups and overgroups; direct products and graph products more generally; acylindrical graphs of groups; relatively hyperbolic overgroups with the given groups as peripheral subgroups.  A  concrete example is the class one gets starting with hyperbolic groups.

As a concrete application of the results in this section, we analyse $3$--manifolds and free-by-cyclic groups; for instance, we use structural results about mapping tori, due to many authors, to apply our combination theorems to prove the following:

\begin{corintro}[Corollary~\ref{cor:free-by-Z}]\label{corintro:free-by-Z}
Let $F$ be a finite-rank free group and $\Phi\in\mathrm{Out}(F)$.  Then $G=F\rtimes_\Phi\integers$ satisfies the uniform power alternative.
\end{corintro}

In \cite[Theorem B]{ABKSV}, it is shown that if $\Phi$ is fully irreducible and $x,y\in G$, then $\langle x,y\rangle$ is either free, free abelian, the Klein bottle group, or has finite index in $G$, all of which imply the uniform power alternative.  It does not appear that the results in \cite{ABKSV} imply that conclusion without the full irreducibility hypothesis, however, so the above corollary seems new.
Examining the constants in the proof shows that the uniform exponent depends only on $\mathrm{rk}(F)$ provided the relatively hyperbolic structure from \cite{DahmaniLi} can be chosen so that the hyperbolicity and acylindricity parameters are bounded in terms of $\mathrm{rk}(F)$. Such bounds do not immediately appear in  \cite{DahmaniLi}, but Fran\c{c}ois Dahmani has explained in personal communication that there are good heuristic reasons to think they should exist.

\subsection{Organisation of this paper} In Section~\ref{sec:power alternative} we introduce the power alternative and its uniform version. Section~\ref{sec:tree-action-criterion} is devoted to proving our criterion for proving that a group satisfies the power alternative, that is, Theorem~\ref{thmintro:WPA for actions on trees}. In Section~\ref{sec:wpa for Artin groups expressed as a visual splitting} we list some examples (and non-examples) of groups that satisfy the power alternative with a particular focus on Artin groups expressed as a visual splitting. In Section~\ref{sec:uwpa for (2,2)-free triangle-free Artin groups} we prove that $(2,2)$-free triangle-free Artin groups satisfy the uniform power alternative, with uniform exponent depending only on the set of labels of a defining graph. Finally, in Section~\ref{sec:recovering} we study the stability of the power alternative under the construction of relatively hyperbolic groups. In particular, we prove that relatively hyperbolic groups satisfy the power alternative, if every peripheral subgroup does. 

\subsection{Acknowledgements} We would like to thank María Cumplido for helpful discussions on the normalisers of parabolic subgroups of Artin groups and Fran\c{c}ois Dahmani for explaining relative hyperbolicity constants in mapping tori.

\section{Preliminaries on the Power Alternative}
\label{sec:power alternative}

In this section, we recall the definition of the power alternative and introduce its uniform version, and prove some preliminary results about them that are used throughout this paper.

\begin{defi}[power alternative]
    \label{def:wpa}
    Let $G$ be a group. 
    \begin{enumerate}
        \item We say that $G$ satisfies the \emph{power alternative} if, for every $g,h\in G$, there exists $n\in\N_{\ge1}$ such that either $g^n$ and $h^n$ commute or $\Span{g^n,h^n}_G\cong F_2$;
        \item Let $n\in\N_{\ge1}$. We say that $G$ satisfies the \emph{$n$-uniform power alternative} or the \emph{power alternative with uniform exponent~$n$} if, for every $g,h\in G$, either $g^n$ and $h^n$ commute or $\Span{g^n,h^n}_G\cong F_2$;
        \item We say that $G$ satisfies the \emph{uniform power alternative} if $G$ satisfies the $n$-uniform power alternative for some $n\in\N_{\ge1}$.
    \end{enumerate}
\end{defi}

It is straightforward, yet very useful, to observe that the power alternative can be detected from finite-index subgroups:

\begin{lem}
    \label{lem:wpa virtual property}
    Let $G,H$ be groups, let $G'\le G$ be a finite-index subgroup isomorphic to a subgroup of $H$. 
    \begin{enumerate}
        \item If $H$ satisfies the power alternative, then $G$ satisfies the power alternative;
        \item If $H$ satisfies the power alternative with uniform exponent $n$, then $G$ satisfies the power alternative with uniform exponent $n\cdot\abs{G:G'}$.
    \end{enumerate}
\end{lem}

\begin{example}
    \label{ex:group with wpa but not uwpa}
    The uniform power alternative is strictly stronger than the non-uniform one. For instance, let $G$ be an infinitely generated torsion group with the property that, for every $n\in\N$, there exists an element of $G$ of order $n$. Then $G$ satisfies the power alternative but with no uniform exponent in general. 
\end{example}

Also observe that the (uniform) power alternative is stable under direct products:

\begin{lem}[stability under direct product]
    \label{lem:wpa stable direct product}
    Let $G_1$ and $G_2$ be groups. 
    \begin{enumerate}
        \item If $G_1$ and $G_2$ satisfy the power alternative, then so does $G_1\times G_2$;
        \item If $G_1$ satisfies the $m$-uniform power alternative and $G_2$ satisfies the $n$-uniform power alternative, then $G_1\times G_2$ satisfies the power alternative with uniform exponent $\operatorname{lcm}(m,n)$.
    \end{enumerate}
\end{lem}

\begin{proof}
    Let $G_1$ and $G_2$ satisfy the power alternative. For $i\in\{1,2\}$, let $\pi_i\colon G_1\times G_2\to G_i$ denote the projection on the $i$-th factor. Let $(g_1,g_2),(h_1,h_2)\in G_1\times G_2$. If there exist $i\in\{1,2\}$ and $n\in\N$ such that $\Span{g_i^n,h_i^n}_{G_i}\cong F_2$, then $\Span{(g_1,g_2)^n,(h_1,h_2)^n}_{G_1\times G_2}$ is non-abelian free as well. If there are $n_1,n_2\in\N_{\ge1}$ such that $[g_1^{n_1},h_1^{n_1}]=[g_2^{n_2},h_2^{n_2}]=1$, then, by choosing $n$ as the least common multiple of $n_1$ and $n_2$, we have that $(g_1,g_2)^n$ and $(h_1,h_2)^n$ commute. The uniform case is analogous.
\end{proof}

\section{Groups acting on trees and the power alternative}
\label{sec:tree-action-criterion}

In this section, we prove a geometric criterion for a group acting on a tree to satisfy the power alternative, namely the stabilisation property from Definition \ref{def:goodname property}.

\begin{thm}
    \label{thm:wpa for action on real tree}
    Let $G$ be a group acting by isometries on a real tree $T$,
    such that the action satisfies the stabilisation property.
    If stabilisers of all points of $T\cup\partial T$ 
    satisfy the power alternative, then so does~$G$.
\end{thm}

Theorem~\ref{thmintro:WPA for actions on trees}, which is stated for our main case of interest, i.e. actions on simplicial trees by graph automorphisms, follows from Theorem \ref{thm:wpa for action on real tree}:

\begin{proof}[Proof of Theorem~\ref{thmintro:WPA for actions on trees}]
    In order to show that the action $G\actson T$ satisfies the hypotheses of Theorem~\ref{thm:wpa for action on real tree}, it suffices to show that stabilisers of points of~$T$ satisfy the power alternative. Let $x\in T$. Then $\stab{G}{x}$ has a subgroup of index at most $2$ that sits as a subgroup of a vertex-stabiliser \cite{serre2002trees}*{Section~3.1}, so Lemma~\ref{lem:wpa virtual property} and Theorem~\ref{thm:wpa for action on real tree} imply the theorem.
\end{proof}

\subsection{Action on trees and fixed point sets} We recall the definition of real trees and establish some graph-theoretic notation. \par 
Recall that, if $X$ is a metric space, a \emph{geodesic} in $X$ is an isometric embedding $\gamma\colon I\to X$, for some interval $I\subseteq\mathbb R$. With a little abuse of notation, we identify $\gamma$ with its image. Following the literature, if $I=\mathbb R$, we call $\gamma$ a \emph{geodesic line}; if $I$ is of the form $]a,+\infty[$ or $[a,+\infty[$, we call $\gamma$ a \emph{geodesic ray}; if $I$ is of the form $[a,b]$, we call $\gamma$ a \emph{geodesic segment}. 

\begin{defi}[real trees]
    A metric space~$T$ is a \emph{real tree} (or $\R$-tree) if the following conditions hold:
    \begin{enumerate}
        \item for each pair of points $x,y\in T$ there exists a unique geodesic segment $[x,y]$ connecting them;
        \item for all points $x,y,z\in T$, if $[x,y]\cap[y,z]=\{y\}$, then $[x,y]\cup[y,z]=[x,z]$.
    \end{enumerate}
\end{defi}

As a standard example, a simplicial tree, together with its path metric where every edge is given length $1$, is a real tree.  For a survey on non-simplicial examples, see \cite{Bestvina:r-trees}. A real tree $T$ is in particular Gromov-hyperbolic, and we denote as usual by $\partial T$ its Gromov boundary.

\medskip

Isometries of a real tree $T$ are either loxodromic or elliptic (see e.g. \cite[1.3]{CullerMorgan}). For a loxodromic isometry $g$, we denote by $\mathrm{Axis}(g)$ the unique axis of $g$, and by $\tau(g)>0$ the translation length of $g$ on its axis. For an elliptic isometry $g$, we denote by $\mathrm{Fix}(g)$ the corresponding fixed-point set, which is a $\langle g\rangle$-invariant and closed convex subset of~$T$.

\medskip 

The following notion will play a key role when studying the subgroups generated by large powers of two isometries of $T$:

\begin{defi}[stable fixed-point set]
  If $g\in G$ is an elliptic element for the action $G\actson T$, then the \emph{stable fixed-point set} of $g$ is defined to be
\[
\stabfix g=\bigcup_{x\in \langle g \rangle - \{1\}}\fix{x},
\]
which is a $\langle g\rangle$-invariant and convex subset of~$T$.  
\end{defi}

We introduce a property that allows us to get some control on stable fixed-point sets.

\begin{defi}
    [\goodname\ property]
    \label{def:goodname property}
    Let~$T$ be a real tree and let $G$ be a group acting on~$T$ via isometries; we say that the action $G\actson T$ has the \emph{\goodname} property if the following condition holds: Let $\{\gamma_n\}_{n\geq 1}$ be a sequence of geodesics of $T$ that is monotonic for the inclusion, that is, $\gamma_n\subset \gamma_{n+1} $ for all $n$, or $\gamma_{n+1}\subset \gamma_{n}$ for all $n$. For each $n\geq 1$, let $\ell(\gamma_n)\in \mathbb{R}_{\geq 0}\cup \{+\infty\}$ denote the length of $\gamma_n$, and let $H_n$ denote the point-wise stabiliser of $\gamma_n$. If $\ell(\gamma_n) \rightarrow +\infty$, then the sequence $\{H_n\}_{n\geq 1}$ of subgroups stabilises, that is, there exists an integer $N\geq 1$ such that $H_n = H_{n+1}$ for all $n \geq N$.
\end{defi}

\begin{example}\label{ex:acylindrical}
    If a group $G$ acts \textit{acylindrically} (see Definition~\ref{defn:acylindricity}) on a real  tree~$T$, then the action has the stabilisation property, since there exists a uniform constant~$L$ such that any geodesic segment of length at least~$L$ has finite pointwise stabiliser. 
\end{example}

\begin{example}\label{exmpl:BS}
    For an example of an action that does not satisfy the stabilisation property, consider the Baumslag--Solitar group $\operatorname{BS}(1,2) = \langle a, b ~|~ bab^{-1}=a^2\rangle$, expressed as the HNN extension $\langle a \rangle*_{\langle b \rangle}$, with Bass--Serre tree denoted $T$. Then the geodesic ray $\gamma$ of $T$ spanned by the vertices $b^{-n}\langle a \rangle$, for $n \geq 0$, has pointwise stabiliser $\langle a \rangle$, and for every $n \geq 1$, the translate $b^n\gamma \supset \gamma$ has pointwise stabiliser $\langle a^{2^n}\rangle$.  
\end{example}

\subsection{Proof of Theorem \ref{thm:wpa for action on real tree}}
In this subsection we will prove Theorem~\ref{thm:wpa for action on real tree} by analysing the dynamics of elements $g$ and~$h$ on the tree~$T$: for a pair $g,h\in G$, we shall study the subgroup $\Span{g^n,h^n}_G$, depending on whether $g$ and $h$ act on~$T$ both elliptically, both loxodromically or in a mixed fashion.

\begin{lem}[elliptic-elliptic case]
 \label{lem:elliptic-elliptic general case}
Let the group $G$ act by isometries on the real tree $T$.  Suppose $g,h\in G$ act elliptically, and that $\stabfix g\cap \stabfix h=\emptyset$.  Then $\langle g,h\rangle\cong \langle g\rangle *\langle h\rangle$.  
\end{lem}

We will use the following elementary observation:

\begin{lem}\label{lem:concatenation_geodesics_distinct}
    Let $T$ be a real tree, and consider a closed path in $T$ that decomposes as the concatenation  $\gamma_1\cup \cdots \cup \gamma_k$  of $k$ geodesic segments $\gamma_1, \ldots, \gamma_k$ of $T$ (concatenated in the given order). Suppose that the overlaps between consecutive $\gamma_i$'s are such that for every $1\leq i \leq k$, we have 
    $$\gamma_i' \coloneqq \gamma_i - ( \gamma_{i-1} \cup \gamma_{i+1}) \neq \varnothing $$
    (with the convention that $\gamma_0 = \gamma_{k+1} = \varnothing$). Then the geodesic segments $\gamma_i$ are pairwise distinct (as subsets of~$T$).
\end{lem}

\begin{proof}
    We leave it to the reader to check by induction on $k$ that the closures of the~$\gamma_i'$ concatenate into a non-backtracking path of $T$ (hence a geodesic of the real tree~$T$), and that for each $1 \leq i \leq k$, the geodesic $\gamma_i'$ belongs to $\gamma_i - \cup_{j \neq i}\gamma_j$. 
\end{proof}

\begin{proof}[Proof of Lemma~\ref{lem:elliptic-elliptic general case}]
Let~$\gamma$ be the unique geodesic in~$T$ between the closed sets~$\fix g$ and~$\fix h$.
    Let~$w=w_0w_1\cdots w_k$ be a non-empty reduced word in the alphabet~$\{g,h\}$, where $w_0=1$ and the $w_i$'s are alternating non-trivial elements of  $\langle g\rangle$ or~$\langle h\rangle$ for $i\in\{1,\dots,k\}$. To prove that $\langle g, h \rangle = \langle g \rangle * \langle h \rangle$, it is enough to show that $w$ acts as a non-trivial isometry of~$T$. 
    For each $i$, let 
    \[
    \gamma_i=\bigcup_{i=0}^k w_0\cdots w_i\cdot \gamma
    \]
    We claim that for every $i$, the interval $\gamma_i - ( \gamma_{i-1} \cup \gamma_{i+1}) $ is non-empty. Indeed, up to translation, this amounts to checking that the interval $\gamma - (\fix{\alpha_i}\cup \fix{\beta_i})$ is non-empty for some suitable non-trivial elements $\alpha_i\in \langle g \rangle$ and $\beta_i \in \langle h \rangle$, and this follows from the assumption that the stable fixed-point sets are disjoint. It now follows from Lemma~\ref{lem:concatenation_geodesics_distinct} that the $\gamma_i$ are pairwise distinct. In particular, $\gamma_0 = \gamma$ and $\gamma_k = w\cdot \gamma$ are distinct, hence $w$ acts non-trivially on $T$. See Figure~\ref{fig:elliptic disjoint fps} for an intuition.
\end{proof}

\begin{figure}
    \centering
    \includegraphics[width=0.5\linewidth]{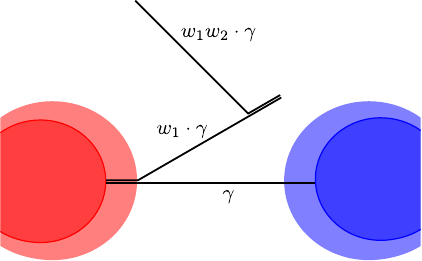}
    \caption{In red and blue the fixed-point sets of $g$ and $h$, respectively. Shaded, the stable fixed-point sets.}
    \label{fig:elliptic disjoint fps}
\end{figure}

\begin{lem}[mixed, bounded case]\label{lem:mixed-bounded}
    Let~$T$ be a real tree and let $G$ be a group acting on~$T$ by isometries. If $g\in G$ acts elliptically on~$T$, $h\in G$ acts loxodromically on~$T$ and $\operatorname{Fix}^{\infty}(g)\cap\Axis h$ is bounded (including the case where it is empty), then there is $n\in\N$ such that
    \[
    \Span{g,h^n}_{G}\cong  \Span{g}_G*\Span{h^n}_{G}.
    \]
\end{lem}

\begin{proof}
Let $\pi:T\to\Axis h$ be the closest-point projection map and let $A=\pi(\stabfix g)$, which is either a single point or equal to $\stabfix g\cap\Axis h$ if the latter is nonempty.  Let $K=\abs{A}$, which is finite by assumption.  

Let $z\in \fix g$ and let $p=\pi(z)$ and let $D=d_T(z,p)$.  For any $k\neq 0$, consider $\pi(g^k\Axis h)$.  Either this is a single point, or it is equal to $g^k\Axis{h}\cap \Axis{h}=B$.  In the latter case, let $x\in B$.  Then $[z,x]=[g^kz,x]$ passes through $p$ since $x\in \Axis{h}$ and through $g^kp$ since $x\in g^k\Axis h$ and $g^kp$ is the projection of $g^kz=z$ onto $g^k\Axis h$.  Since $d_T(z,g^kp)=D$, we have $p=g^kp$, i.e. $p\in \fix{g^k}\cap \Axis{h}$.  Thus $g^{2k}x=x$ and hence $x\in \stabfix g$.  We have shown that $\pi(g^k\Axis h)$ has diameter at most $K$ for all $k\neq 0$, and moreover it is contained in the $K$--ball about the point $p$, which is independent of $k$.  

Now consider $\partial T$, equipped with the usual visual metric.  We have just shown that there are open sets $U^\pm\subset \partial T$, respectively containing $h^\pm$, such that $g^k\{h^\pm\}\cap (U^+\cup U^-)=\emptyset$ for $k\neq 0$.  

On the other hand, $\langle h\rangle$ acts on $\partial T$ with north-south dynamics, so there exists $n>0$ such that $h^{kn}(\partial T-(U^-\cup U^+))\subset U^-\cup U^+$ for all $k\neq 0$.  We can therefore apply the ping-pong lemma (see e.g. \cite[II.B.24]{delaHarpe}) to conclude.
\end{proof}

\begin{lem}[loxodromic-loxodromic, bounded case]\label{lem:lox-lox}
    Let~$T$ be a real tree and let $G$ be a group acting on~$T$ by isometries. Let~$g,h\in G$ be elements acting loxodromically on~$T$. If~$\Axis{g}\cap\Axis{h}$ is bounded (possibly empty), then there exists $n\in\N$ such that $\Span{g^n,h^n}_{G}\cong F_2$.
\end{lem}

\begin{proof}
This is a direct consequence of, for instance, \cite[Lemma 2.6]{CullerMorgan}.
\end{proof}

\begin{figure}
    \centering
    \includegraphics[width=0.6\textwidth, height=0.15\textheight]{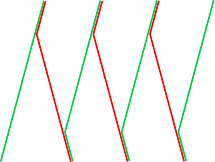}
    \caption{In green, the axis of $g^n$ and its translates; in red, the axis of $h^n$ and its translates.}
    \label{fig:loxo-loxo}
\end{figure}

\begin{lem}[loxodromic-loxodromic, unbounded case]\label{lem:lox-unbounded}
    Let~$T$ be a real tree and let $G$ act on~$T$ by isometries. If $g,h\in G$ both act loxodromically on~$T$ and $\Axis{g}\cap\Axis{h}$ is unbounded, then 
    $g$ and $h$ fix a common point in $\partial T$.
\end{lem}

 \begin{proof}
 This is immediate.
 \end{proof}

\begin{lem}[mixed, unbounded case]\label{lem:mixed-unbounded}
    Let~$T$ be a real tree and let $G$ act on~$T$ under the hypotheses of Theorem~\ref{thm:wpa for action on real tree}. If $g,h\in G$, with $g$ acting elliptically on~$T$ and $h$ acting loxodromically on~$T$, and with  $\stabfix{g}\cap\Axis{h}$  unbounded, then there exists $n\in\N_{\ge1}$ such that 
    $g^n$ and $h$ fix a common point in $\partial T$. 
\end{lem}

\begin{proof}
The stabilisation property applied to the sequence of geodesics $\{(\bigcup^k_{i=1}\fix {g^i}) \cap \Axis{h}\}_n$ 
implies that there exists an integer $k$ such that $$(\bigcup^k_{i=1}\fix {g^i}) \cap \Axis{h} = \stabfix g \cap  \Axis{h}$$ Since $\bigcup^k_{i=1}\fix {g^i} \subset \fix {g^{k!}}$, we get in particular that $\fix {g^{k!}} \cap \Axis{h} = \stabfix g \cap  \Axis{h}$. Since this intersection is unbounded by assumption, it follows that $g^{k!}$ and $h$ fixes a point in $\partial T$. Thus, $g^{k!}$ and $h$ belong to the stabiliser of a common boundary point.
\end{proof}

Combining the lemmas gives:

\begin{prop}\label{prop:free-or-common-stabiliser}
Let $G$ act by isometries on the real tree $T$ and suppose that the action satisfies stabilisation. For all $g,h\in G$, there exists $n\in\N_{\ge1}$ such that either $\langle g^n,h^n\rangle_G\cong F_2$, or there exists $x\in T\cup\partial T$ such that $g^n,h^n\in\stab{G}{x}$. 
\end{prop}

\begin{proof}
Each of $g$ and $h$ is either elliptic or loxodromic \cite[1.3]{CullerMorgan}. By Lemmas~\ref{lem:mixed-unbounded}, \ref{lem:lox-unbounded}, \ref{lem:lox-lox}, and \ref{lem:mixed-bounded}, if $h$ is loxodromic, then either $\langle g^n,h^n\rangle_G\cong F_2$ for suitable $n$, or $g^n$ and $h^n$ stabilise a common point in $\partial T$.  Otherwise we can conclude using Lemma~\ref{lem:elliptic-elliptic general case}.
\end{proof}

Theorem~\ref{thm:wpa for action on real tree} is now a direct consequence of Proposition~\ref{prop:free-or-common-stabiliser}.

\section{Examples and first application}
\label{sec:wpa for Artin groups expressed as a visual splitting}

We collect here some first direct applications of Theorem~\ref{thm:wpa for action on real tree}, together with examples and non-examples of groups that satisfy the power alternative. In the last part of the section we will turn our focus on Artin groups that admit a visual splitting, as a more detailed example. 

\subsection{First applications}\label{subsec:first applications}

To illustrate a typical case of Theorem \ref{thm:wpa for action on real tree}, we have the following statement about acylindrical actions on trees (we refer the reader to Definition \ref{defn:acylindricity} for the definition of acylindricity).

\begin{cor}\label{cor:acyl}
    Let $T$ be a real tree and let $G$ act on $T$ acylindrically. If stabilisers of points of $T$ satisfy the power alternative, then so does $G$. 
\end{cor}

\begin{proof}
By \cite[Theorem 1.1]{Osin:acyl}, stabilisers of points in $\partial T$ either stabilise points of $T$ or are virtually $\integers$, and in particular satisfy the power alternative. Since the action satisfies the stabilisation property by Example~\ref{ex:acylindrical}, Theorem \ref{thm:wpa for action on real tree} applies.
\end{proof}

In fact, the lemmas used to prove Theorem \ref{thm:wpa for action on real tree} actually give a more general statement without further work; the statement requires a preliminary definition.

Recall what it means for a group $G$ to \emph{satisfy a law}.  Let $F(a_1,\ldots,a_r)$ be a free group on the generators $a_1,\ldots,a_r$ and let $w\in F(a_1,\ldots,a_r)$ be a nontrivial reduced word.  Define a map $w:G^r\to G$ as follows: given $g_1,\ldots,g_r\in G$, let $w(g_1,\ldots,g_r)$ be the image of $w$ under the homomorphism $F_r\to G$ induced by $a_i\mapsto g_i$ for $i\in\{1,\ldots,r\}$.  Then $G$ \emph{satisfies the law $w$} if $w(g_1,\ldots,g_r)=1$ for all $g_1,\ldots,g_r\in G$.  

\begin{defi}[law-power alternative]\label{def:law}
Given $w\in F_r-\{1\}$, the group $G$ satisfies the \emph{$w$--law-power alternative} if, for all $g,h\in G$, there exists $n\in\N_{\geq 1}$ such that either $\langle g^n,h^n\rangle_G\cong F_2$ or $\langle g^n,h^n\rangle_G$ satisfies the law $w$.
\end{defi}

\begin{rem}\label{rem:law-remarks}
The power alternative is the  $w$--law-power alternative for $w=[a_1,a_2]\in F(a_1,a_2)$.  One can define uniform versions of the $w$--law-power alternative, in analogy to Definition \ref{def:wpa}.  
\end{rem}

\begin{cor}\label{cor:law-alternative}
Fix $w\in F_r-\{1\}$. Suppose $G$ acts on the real tree $T$ satisfying stabilisation and that $\stab{G}{x}$ satisfies the $w$--law-power alternative for all $x\in T\cup\partial T$.  Then $G$ satisfies the $w$--law-power alternative.
\end{cor}

\begin{proof}
This follows from Proposition \ref{prop:free-or-common-stabiliser} and Definition \ref{def:law} along with our assumption on stabilisers.
\end{proof}

Before our remaining applications, we need to recall the structure of stabilisers of boundary points.

\subsection{Ultimate translation length}\label{subsec:UTL}
Let us introduce the concept of ultimate translation length for an isometry of the real tree $T$ fixing a point at infinity. 

\begin{defi}[ultimate translation length]
    Let~$T$ be a real tree and let~$G$ be a group acting on~$T$ by isometries; the \emph{ultimate translation length} of~$g\in\stab{G}{\xi}$ towards~$\xi$ is defined as
    \[
    \utl{g}=
    \begin{cases}
        \tau(g)         &\text{if~$g$ acts loxodromically on~$T$ and~$\xi=g^{+\infty}$}\\
        0               &\text{if~$g$ acts elliptically on~$T$}\\
        -\tau(g)        &\text{if~$g$ acts loxodromically on~$T$ and~$\xi=g^{-\infty}$}
    \end{cases}
    \]
\end{defi}

\begin{lem}[{\cite[Corollary 2.3]{CullerMorgan}}]
    \label{lem:ses for utl}
    Let~$T$ be a real tree, let $G$ be a group acting on~$T$ by isometries and let $\xi\in\partial T$. The function
    \begin{align*}
        \operatorname{utl}_\xi\colon\stab{G}{\xi}&\to\R\\
        g&\mapsto\utl{g}
    \end{align*}
    defines a group homomorphism.  If $T$ is simplicial, then the image of $\operatorname{utl}_\xi$ is in $\Z$.
\end{lem}

Note that an element of $\stab{G}{\xi}$ has zero ultimate translation length towards the point at infinity~$\xi$ if and only if it fixes pointwise a geodesic ray converging to~$\xi$. Therefore 
\[
\Ker{\operatorname{utl}_\xi}=\cup\big\{\pstab{G}{\gamma}:\text{$\gamma$ geodesic ray in~$T$ converging to~$\xi$}\big\}.
\]

\begin{lem}
    \label{lem:ker of utl is a union}
    Let~$T$ be a real tree and let $G$ be a group acting on~$T$ by isometries. Let $\{I_n\}_{n\in\N}$ be a sequence of geodesic rays converging to~$\xi$ such that:
    \begin{enumerate}
        \item for every~$n\in\N$, $I_{n+1}\subsetneq I_n$;
        \item $\bigcap_{n\in\N}I_n=\varnothing$;
    \end{enumerate}
    Then $\Ker{\operatorname{utl}_\xi}=\bigcup_{n\in\N}\pstab{G}{I_n}$.
\end{lem}

\begin{proof}
    Because~$\{I_n\}_{n\in\N}$ is a family of geodesic rays converging to~$\xi$, the inclusion
    \[
    \bigcup_{n\in\N}\pstab{G}{I_n}\subset\Ker{\operatorname{utl}_\xi}
    \]
    is straightforward. For the reverse inclusion, let~$g\in\Ker{\operatorname{utl}_\xi}$, so that there exists a geodesic ray~$\gamma\subset T$ converging to~$\xi$, which~$g$ fixes pointwise. Because the sequence of the~$I_n$'s is inclusion-wise decreasing and satisfies $\bigcap_{n\in\N}I_n=\varnothing$, there exists~$n\in\N$ such that $I_n\subset \gamma$.
\end{proof}

\subsection{Combination theorem, examples, and non-examples}\label{subsec:technical-combination-theorem}
The next statement will be used in conjunction with results in Section \ref{sec:recovering} about mapping tori.

\begin{cor}\label{cor:law-with-mapping-tori}
Let $w\in F_r-\{1\}$.  Suppose $G$ acts by graph automorphisms, without inversions, on the simplicial tree $T$ and the action satisfies stabilisation.  Let $\mathfrak C$ be the class of exactly those $C\leq G$ such that $C\in \mathfrak C$ is the intersection of finitely many vertex stabilisers.  For every $z\in G$ and $C\in\mathfrak C$ such that $C\trianglelefteq \langle C,z\rangle_G=:H$, suppose that $H$ satisfies the $w$--law-power alternative.  Then $G$ satisfies the $w$--law-power alternative.
\end{cor}

\begin{proof}
Let $g,h\in G$.  By Proposition \ref{prop:free-or-common-stabiliser}, either $\langle g^n,h^n\rangle_G\cong F_2$ for some $n\in\N_{\ge1}$, or $n$ can be chosen so that $g^n,h^n\in\stab{G}{x}$ for some $x\in T\cup\partial T$.  If $x\in T$, then our hypothesis on stabilisers applies.  Therefore suppose that $x\in \partial T$.  Consider the homomorphism $\operatorname{utl}_x\colon\stab{G}{x}\to\Z$ from Lemma \ref{lem:ses for utl}.  If $\utl{g^n}=\utl{h^n}=0$, then $g^n,h^n$ stabilise a common point in $T$, and we are done.  So we may assume that $\operatorname{utl}_x$ has nontrivial image in $\Z$.  By Lemma \ref{lem:ker of utl is a union} and \goodname, $C=\Ker{\operatorname{utl}_x}\in\mathfrak C$, so we have a short exact sequence $1 \rightarrow C\hookrightarrow \stab{G}{x}\twoheadrightarrow \integers \rightarrow 1$.  Hence $\stab{G}{x}$ satisfies the $w$--law-power alternative, as required.
\end{proof}

As a sample application, we have:

\begin{cor}\label{cor:graph-of-free-groups}
Let $G$ split as a finite graph of groups where:
\begin{itemize}
    \item Each vertex group $G_v$ is of one of the following types: $G_v$ is free; $G_v$ is the fundamental group of a closed hyperbolic surface; $G_v$ is a one-ended hyperbolic group with $\operatorname{Out}(G_v)$ finite.
    \item Each edge group $G_e$ is free or a surface group, and quasiconvex in the incident vertex groups.
    \item The action of $G$ on the Bass-Serre tree satisfies \goodname.
\end{itemize}
Then $G$ satisfies the power alternative.
\end{cor}

\begin{proof}
Let $\mathfrak C_0$ be the set of subgroups $\stab{G}{v}$ for $v\in\Ver T$.  Let $\mathfrak C$ be the set of subgroups of $G$ arising as finite intersections of elements of $\mathfrak C_0$.  Let $C_1,\ldots,C_n\in\mathfrak C_0$ and let $v_i$ be a vertex of $T$ with $\stab{G}{v_i}=C_i$. 

Let $C=\bigcap_{i=1}^nC_i$.  We claim that $C$ is hyperbolic.  Indeed, if $n=1$, this holds by hypothesis; in this case it is immediate that $C=C_n$ is quasiconvex in $C_n$.

Assume by induction that for all $J\subsetneq\{1,\ldots,n\}$, the subgroup $C^J=\bigcap_{i\in J}C_i$ is quasiconvex in $C_i$ for all $i\in J$.  Let $T'\subset T$ be the finite tree which is the convex hull of $\{v_1,\ldots,v_n\}$.  By relabelling, we can assume $v_n$ is a leaf of $T'$.  Let $T''$ be the convex hull of $\{v_1,\ldots,v_{n-1}\}$.  Let $e_1,e_2,\ldots,e_k$ be the sequence of consecutive edges in a shortest path in $T'$ from $v_n$ to an arbitrary vertex in $\{v_1,\ldots,v_{n-1}\}$, which we may assume is $v_{n-1}$.  Let $C'=C^{\{v_1,\ldots,v_{n-1}\}}$.  So, by the induction hypothesis, $C'$ is quasiconvex in $C_{n-1}$ (and hence hyperbolic).  Now, $C=C'\cap \bigcap_{j=1}^k \stab{G}{e_j}$.  Since edge-groups are intersections of their incident vertex groups, are quasiconvex in those vertex groups by assumption, and since finite intersections of quasiconvex subgroups of a hyperbolic group are quasiconvex, it follows that $C$ is quasiconvex in $C_n$.  In particular, each $C\in\mathfrak C$ is hyperbolic.  By Theorem \ref{thm:hyperbolic} below (which is a well-known fact about hyperbolic groups; see e.g. \cite{loh2017geometric}), each $C\in\mathfrak C$ satisfies Wise's power alternative.

Now let $C\in\mathfrak C_0$ and let $H=\langle C,z\rangle_G$ where $z\in G$ normalises $C$.  So, $H$ is the mapping torus of some automorphism of $C$.  If $C$ is the intersection of at least two vertex groups, our hypotheses ensure it is free or a surface group, so either Corollary \ref{cor:free-by-Z} or Corollary \ref{cor:3-manifolds} implies the power alternative for $H$.  Otherwise, $C$ is a vertex group, so either the same argument applies, or $\operatorname{Out}(C)$ is finite.  In the latter case, $H$ is virtually a direct product of $\integers$ with a finite-index subgroup of $C$, so Lemma \ref{lem:wpa virtual property} and Lemma~\ref{lem:wpa stable direct product} imply that $H$ satisfies the power alternative.  Now the result follows from Corollary \ref{cor:law-with-mapping-tori}.   
\end{proof}

\begin{rem}\label{rem:JSJ-question}
Corollary \ref{cor:graph-of-free-groups} is intended mainly to illustrate the technical conditions in Corollary \ref{cor:law-with-mapping-tori}, so we have probably not given the strongest possible statement.  Specifically, one can probably allow more general vertex and edge groups, at the expense of complicating the proof in a tangential way.  Indeed, if $C\in\mathfrak C$ is allowed to be a hyperbolic group that is not free or a surface group, then one can presumably consider first the Grushko decomposition of $C$, and then a JSJ decomposition of each one-ended vertex group, to analyse the subgroup $\langle C,z\rangle_G\cong C\rtimes \integers$.  In other words, one can likely make a stronger statement if equipped with a generalisation of Corollary \ref{cor:3-manifolds} and Corollary \ref{cor:free-by-Z} for hyperbolic-by-cyclic groups.
\end{rem}

Now we turn to negative examples, all involving commensuration of edge groups in one way or another.

\begin{example}\label{exmp:baum-sol}
Let $G$ be the Baumslag-Solitar group $\basol(1,n)$.  The action of $G$ on the Bass-Serre tree~$T$ of the usual splitting as an ascending HNN extension of $\integers$ fixes a point in $\partial T$, so any statement, like Theorem \ref{thm:wpa for action on real tree}, that assumes the power alternative for stabilisers of boundary points cannot give useful information.  And in fact $G$ does not satisfy the power alternative when $|n|>1$, since $[G,G]$ is isomorphic to the additive group of $n$--adic rationals and hence $G$ satisfies the law $[[a,b],[c,d]]=1$ over $F_4$.
\end{example}

\begin{example}\label{exmp:leary-minasyan}
Suppose that $G$ is a finite graph of groups whose vertex groups are finite-rank free abelian groups and whose edge groups have finite index in the incident vertex groups.  Groups of this type with many interesting properties were constructed by Leary-Minasyan in \cite{leary2021commensurating}.

For such $G$, the Bass-Serre tree $T$ is locally finite and $G$ acts on $T$ cocompactly.  Let $\mathfrak C$ be the class of finite-rank free abelian subgroups of $G$.  Obviously $\mathfrak C$ is closed under intersections and elements of $C$ trivially satisfy the power alternative.  

We consider the specific example $G=G_{k,m}$ from \cite[Example 9.4]{leary2021commensurating}, with $-2m<k<2m$ and $k\not\in\{0,\pm m\}$.  Such a $G$ does not satisfy the power alternative, by \cite[Corollary 9.6]{leary2021commensurating}.  

Then by \cite[Corollary 9.3]{leary2021commensurating}, $G$ is a \cat 0 group, and in fact, by \cite[Theorem 7.2 and Theorem 7.5]{leary2021commensurating}, $G$ is a uniform lattice in $\euclidean^2\times T$.  In particular, the vertex and edge stabilisers for the action on $T$ are in $\mathfrak C$.  However, the \goodname~property fails in this case.
\end{example}

\begin{example}\label{exmpl:irreducible-lattice}
Let $T_1,T_2$ be locally finite trees and suppose that $G\leq \operatorname{Isom}(T_1)\times \operatorname{Isom}(T_2)$ is a cocompact irreducible lattice as in \citelist{\cite{BurgerMozes}\cite{Wise:CSC}}.  Then the vertex stabilisers and edge-stabilisers are finite-rank free groups; let $\mathfrak C$ be the class of finite-rank free groups in $G$ that are elliptic on $T_1$.  Then $\mathfrak C$ is closed under finite intersections (the stabiliser of any point in $T_1$ is a commensurated subgroup of $G$).  Irreducibility means that the image of $G$ in $\operatorname{Isom}(T_1)$ is non-discrete \cite[Theorem 1.4]{CapraceMonod}. It is an open question when such $T_1\times T_2$ must contain an \emph{anti-torus}, i.e. an isometrically embedded flat $F$ such that the map $F\hookrightarrow T_1\times T_2\to (T_1\times T_2)/G$ does not factor through a $\pi_1$--injective torus, but $F$ is the convex hull in $T_1\times T_2$ of a pair of axes of elements $a_1,a_2$, with $a_i$ acting hyperbolically on $T_i$ (and elliptically on the other tree).  The projection of an anti-torus to one of the $T_i$ contains a ray witnessing the failure of the \goodname~property.  And the subgroup $\langle a_1,a_2\rangle$ has the property that nontrivial powers of $a_1$ and $a_2$ do not commute.  But it is unknown if $G$ satisfies the power alternative even in the absence of  \goodname.
\end{example}

\subsection{Artin groups and visual splittings}
\label{sec:background on artin groups}

We devote the remaining part of the section to prove a sufficient condition determining when an Artin group with a visual splitting satisfies the power alternative. 

\begin{defi}[Artin groups]
    Let $\Gamma=(\Ver\Gamma,\Edge\Gamma)$ be a finite simplicial graph;
    \begin{enumerate}
        \item A \emph{labelling} of~$\Gamma$ is a map $m\colon\Edge\Gamma\to\N_{\ge2}$. If $\{s,t\}$ is an edge of~$\Gamma$, then we write $m_{st}$ in place of $m(\{s,t\})$. We call the pair $(\Gamma,m)$ a presentation graph;
        \item Let $(\Gamma,m)$ be a presentation graph; the \emph{Artin group} associated to $(\Gamma,m)$ is the group $A_\Gamma$ with presentation
        \[
        \Span{\Ver\Gamma\mid\underbrace{sts\cdots}_{m_{st}}=\underbrace{tst\cdots}_{m_{st}}\text{ for all $\{s,t\}\in\Edge\Gamma$}}
        \]
        \item the Coxeter group associated to $\Gamma$ is the quotient $W_\Gamma$ of $A_\Gamma$ by the normal closure of the subgroup generated by $\{s^2:s\in\Ver\Gamma\}$.
    \end{enumerate}
\end{defi}

We will often ``forget'' about the labelling and refer to $A_\Gamma$ as the Artin group associated to~$\Gamma$. \par 

If~$\Gamma$ is a graph and $\Lambda\le\Gamma$ is a subgraph, recall that~$\Lambda$ is said to be \emph{the full subgraph of~$\Gamma$ induced by $\Ver\Lambda$} (or just a full subgraph) if every edge of~$\Gamma$ connecting two vertices of~$\Lambda$ is also an edge of~$\Lambda$.\par
For every subset $S$ of vertices of~$\Gamma$ we can consider the subgroup~$\Span{S}_{A_\Gamma}$ of $A_\Gamma$ generated by $S$. Such a subgroup is called the \emph{standard parabolic subgroup} of $A_\Gamma$ generated by $S$ and is denoted $A_S$. If~$\Gamma_S$ denotes the full subgraph of~$\Gamma$ spanned by $S$, then $A_S$ is  isomorphic to~$A_{\Gamma_S}$ b y a result of van der Lek~\cite{van1983homotopy}. A subgroup of $A_\Gamma$ that is conjugated to a standard parabolic subgroup is called a \emph{parabolic subgroup} of $A_\Gamma$. The inclusion between parabolic subgroups is well-understood.

\begin{thm}[inclusion of parabolic subgroups, {\cite{blufstein2023parabolic}*{Theorem~1.1}}]
    \label{thm:inclusion of parabolic subgroups}
    Let~$\Gamma$ be a finite labelled simplicial graph and let $X,Y\subset\Ver{\Gamma}$. If $gA_Xg^{-1}\le A_Y$ for some $g\in A_\Gamma$, then there exist $Z\subset Y$ and $h\in A_Y$ such that $gA_Xg^{-1}=hA_Zh^{-1}$.

    In particular, if $gA_Xg^{-1}< A_Y$, then $|X|<|Y|$.
\end{thm}

Group-theoretic results holding for all Artin groups are quite rare and it is common to restrict to certain families of such groups instead. We recall here the definition of the most common ones.

\begin{defi}[families of Artin groups]
    \label{def:families of Artin groups}
    Let~$\Gamma$ be a finite presentation graph. We say that the associated Artin group $A_\Gamma$ is:
    \begin{enumerate}
        \item \emph{dihedral} if $\abs{\Ver\Gamma}=2$; 
        \item \emph{two-dimensional} if, for every $r,s,t$ distinct vertices of~$\Gamma$, 
        \[
        \frac{1}{m_{rs}}+\frac{1}{m_{st}}+\frac{1}{m_{tr}}\le1
        \]
        with the convention that $1/m_{rs}=0$, when $\{r,s\}$ is not an edge of~$\Gamma$;
        \item \emph{spherical} if the associated Coxeter group $W_\Gamma$ is finite;
        \item \emph{of hyperbolic-type} if the associated Coxeter group $W_\Gamma$ is hyperbolic;
        \item \emph{of FC-type} if every clique of~$\Gamma$ induces a spherical parabolic subgroup;
        \item \emph{even} if all the labels of $\Gamma$ are even.
    \end{enumerate}
\end{defi}

Much of our current understanding of Artin groups comes from our understanding of their parabolic subgroups. The following two properties are conjectured to holds for all Artin groups~\cite{godelle2007artin}*{Conjecture~1}.

\begin{defi}[structure of normalisers, intersection of parabolics]
    \label{def:intersection of parabolics, structure of normalisers}
    Let $A_\Gamma$ be an Artin group.
    \begin{enumerate}
        \item For $S\subset\Ver\Gamma$ we define the \emph{quasi-centraliser} of $A_S$ in $A_\Gamma$ as
        \[
        \QZ{A_\Gamma}{A_S}=\{g\in A_\Gamma:gSg^{-1}=S\}
        \]
        We say that $A_\Gamma$ enjoys the \emph{normaliser structure property} if, for every $S\subset\Ver\Gamma$, one has
        \[
        \norm{A_\Gamma}{A_S}=A_S\cdot\QZ{A_\Gamma}{A_S}
        \]
        \item We say that $A_\Gamma$ has the \emph{parabolic intersection property} if, for all parabolic subgroups $P,Q\le A_\Gamma$, the intersection $P\cap Q$ is a parabolic subgroup of $A_\Gamma$.
    \end{enumerate}
\end{defi}

    Godelle proved the normaliser structure property for Artin groups of FC-type and for two-dimensional Artin groups~\citelist{\cite{godelle2003parabolic}\cite{godelle2007artin}}. As for the parabolic intersection property, Blufstein proved it for the family of two-dimensional $(2,2)$-free Artin groups~\cite{blufstein2022parabolic} and Antol\'in and Foniqi proved it for the family of even FC type Artin groups~\cite{antolin2022intersection}. \medskip

\boldparagraph{Visual splittings} 
For an amalgamated product of two groups there is a natural choice of tree to act on, namely the associated Bass-Serre tree. In the context of Artin groups, certain amalgamated product decompositions can be read directly from the defining graph (see for instance~\cite{charney2024acylindrical}).  Let~$\Gamma$ be a presentation graph and let $\Gamma_0\le\Gamma_1,\Gamma_2<\Gamma$ be non-empty induced subgraphs such that $\Gamma_1\cup\Gamma_2=\Gamma$ and $\Gamma_1\cap\Gamma_2=\Gamma_0$. This data gives an isomorphism
\[
    A_{\Gamma_1}*_{A_{\Gamma_0}}A_{\Gamma_2}\cong A_\Gamma
\]
at the level of the associated Artin groups and the amalgamated product $A_{\Gamma_1}*_{A_{\Gamma_0}}A_{\Gamma_2}$ is called a \emph{visual splitting} for $A_\Gamma$. In this section, we will denote by~$T$ the Bass-Serre tree associated with the amalgamation $A_{\Gamma_1}*_{A_{\Gamma_0}}A_{\Gamma_2}$.

\begin{lem}\label{lem:stab_geod_parabolic}
    Let $A_\Gamma$ be an Artin group satisfying the intersection property. Consider the action $A_\Gamma \actson T$ on the Bass--Serre tree corresponding to some visual splitting. Then for every geodesic $\gamma$ of $T$, the pointwise stabiliser of $\gamma$ is a parabolic subgroup of $A_\Gamma$. 
\end{lem}

\begin{proof}
    Stabilisers of vertices of $T$ are parabolic subgroups by definition of a visual splitting. Since the action on $T$ is without inversion, an element of $A_\Gamma$ fixes $\gamma$ pointwise if and only if it fixes every vertex of that geodesic, and the result then follows from the intersection property for $A_\Gamma$.
\end{proof}

\begin{lem}
    \label{prop:action on BS tree has property}
    Let $A_\Gamma$ be an Artin group satisfying the intersection property. Consider the action $A_\Gamma \actson T$ on the Bass--Serre tree corresponding to some visual splitting. Then this action has the \goodname ~property.
\end{lem}

\begin{proof}
    Since stabilisers of geodesics are parabolic subgroups by Lemma~\ref{lem:stab_geod_parabolic}, the result follows from the fact that strict chains of parabolic subgroups have bounded length by Theorem~\ref{thm:inclusion of parabolic subgroups}.
\end{proof}

\boldparagraph{Proof of Corollary~\ref{corintro:wpa visual splitting artin}} The key result for the proof of Corollary~\ref{corintro:wpa visual splitting artin} is the following observation.

\begin{lem}
    \label{lem:stabiliser-by-cyclic have wpa}
    Let $A_\Gamma\cong A_{\Gamma_1}*_{A_{\Gamma_0}}A_{\Gamma_2}$ be an Artin groups expressed as a visual splitting satisfying the hypotheses of Corollary~\ref{corintro:wpa visual splitting artin}. Let $\mathfrak C$ be the set of parabolic subgroups of $A_\Gamma$ that are conjugated with a parabolic subgroup of $A_{\Gamma_1}$ or $A_{\Gamma_2}$. For every $z\in A_{\Gamma}$ and for every $P\in\mathfrak C$ such that $P\trianglelefteq\Span {P,z}=\colon H$,
    either $H$ is virtually $P$ or $H$ contains a finite-index subgroup of the form $P\times\Z$. In particular, $H$ satisfies the power alternative.
\end{lem}

\begin{proof}
    Up to a conjugation in $A_\Gamma$, we can assume that $P$ is a standard parabolic subgroup $A_S$, for some $S\subset\Ver{\Gamma_1}$ or $S\subset \Ver {\Gamma_2}$. Because the power alternative holds for $A_{\Gamma_1}$ and $A_{\Gamma_2}$, every element of $\mathfrak C$ satisfies it. If $z$ is torsion, then $H$ is virtually $A_S$. We can then assume that $z$ has infinite order. 
    By assumption, we have that $z$ normalises $A_S$. By the normaliser structure property for $A_\Gamma$, $\norm{A_\Gamma}{A_S}\cong A_S\cdot\QZ{A_\Gamma}{A_S}$ and, up to a translation in $A_S$, we may assume that $z\in\QZ{A_\Gamma}{A_S}$. Because $S$ is a finite set, a non-trivial power $z^n$ centralises $A_S$. In particular, $H$ (which is a semi-direct product of the form $A_S\rtimes\Z$) has a finite-index subgroup that is isomorphic to a direct product $A_S\times\Z$. The final claim follows by Lemma~\ref{lem:wpa stable direct product}.
\end{proof}

We are now ready to prove Corollary~\ref{corintro:wpa visual splitting artin}.

\begin{proof}[Proof of Corollary~\ref{corintro:wpa visual splitting artin}]
    Let~$T$ be the Bass-Serre tree associated to the visual splitting of $A_\Gamma$. The action satisfies the stabilisation property by Lemma~\ref{prop:action on BS tree has property}. 
    Let $\mathfrak C$ be the family of those subgroups of $A_\Gamma$ that arise as finite intersections of stabilisers of $T$, so, in particular, $\mathfrak C$ consists of parabolic subgroups of $A_\Gamma$ that are conjugated with parabolic subgroups of $A_{\Gamma_1}$ or $A_{\Gamma_2}$. By combining Lemma~\ref{lem:stabiliser-by-cyclic have wpa} and Corollary~\ref{cor:law-with-mapping-tori}, we obtain the claim.
\end{proof}

\begin{rem}\label{rem:graph-prod}
    The approach developed in this section can also be used to recover the power alternative for graph products of groups satisfying the power alternative. Indeed, similarly to Artin groups, there is a well-defined notion of parabolic subgroup and of visual decomposition for these groups. The intersection property for graph products follows from \cite[Corollary~3.6]{antolin2015tits}, and the structure of the normalizer of a parabolic subgroup follows from \cite[Proposition~3.13]{antolin2015tits}. Thus, an application of Corollary~\ref{cor:law-with-mapping-tori}, following a similar proof as above, reduces the power alternative for a given (irreducible) graph product of groups to the power alternative for the various direct products of local groups, and thus to the local groups themselves by Lemma~\ref{lem:wpa stable direct product}.
\end{rem}

Note that every Artin group can be obtained via a sequence of amalgamated products over standard parabolic subgroups, terminating at Artin groups with complete presentation graphs (also called free-of-infinity Artin groups). In particular,  Corollary~\ref{corintro:wpa visual splitting artin} implies the following ``reduction'' theorem, under the hypothesis that the intersection property and the normalizer property hold:

\begin{cor}\label{cor:reduction_free_of_infinity}
    Let $\mathcal{C}$ be a class of Artin groups that is closed under taking standard parabolic subgroups and under taking amalagamated products over standard parabolic subgroups. Suppose that the intersection property and the normalizer property hold for all Artin groups in $\mathcal{C}$. Then if the power alternative holds for all free-of-infinity Artin groups in $\mathcal{C}$,  then it holds for all Artin groups in $\mathcal{C}$. 
\end{cor}

\section{The uniform power alternative for \texorpdfstring{$(2,2)$}{(2,2)}-free triangle-free Artin groups}
\label{sec:uwpa for (2,2)-free triangle-free Artin groups}

In this section, we show that some Artin groups satisfy the uniform power alternative, with a uniform exponent depending only on the set of labels of the defining graph.

\begin{defi}
    Let $A_\Gamma$ be an Artin group with presentation graph $\Gamma$. We say that: 
    \begin{itemize}
        \item $A_\Gamma$ is \emph{triangle-free} if $\Gamma$ does not contain an induced $3$-cycle.
        \item $A_\Gamma$ is \emph{(2,2)-free} if $\Gamma$ does not contain two adjacent edges with label $2$.
    \end{itemize}
\end{defi}

Note that triangle-free Artin groups are in particular $2$-dimensional Artin groups of FC type. Moreover $(2,2)$-free and triangle-free Artin groups are additionally of hyperbolic type. 

\medskip 

Our approach to prove Theorem~\ref{corintro:uniform wpa} is via the following relative notion:

\begin{defi}[relative uniform power alternative]
    Let $N\in\N_{\ge1}$; the Artin group $A_\Gamma$ satisfies the \emph{power alternative relative to complete parabolics} with uniform exponent~$N$ if, for every $g,h\in A_\Gamma$, one of the following conditions hold:
    \begin{enumerate}
        \item $g$ and $h$ belong to a common parabolic subgroup based on a complete subgraph of~$\Gamma$;
        \item $\Span{g^N,h^N}_{A_\Gamma}$ is either abelian or non-abelian free.
    \end{enumerate}
\end{defi}

For the sake of conciseness we will refer to the notion of uniform power alternative relative to complete parabolics just as the relative uniform power alternative. The focus of this section will be to prove the relative uniform power alternative for the class of $(2,2)$-free triangle-free Artin groups:

\begin{thm}[relative uniform power alternative]
    \label{thm:rel wpa artin}
    Let $A_\Gamma$ be an Artin group whose defining graph is $(2,2)$-free and triangle-free. For every $N\in\N_{\ge3}$,  $A_\Gamma$ satisfies the relative power alternative with uniform exponent~$N$.
\end{thm}

If the group $A_\Gamma$ enjoys the relative power alternative and every standard parabolic subgroup of $A_\Gamma$ based on a complete graph satisfies the uniform power alternative, then $A_\Gamma$ inherits the uniform power alternative in the following sense:

\begin{lem}
    Let $N_0\in\N_{\ge1}$ and let the Artin group $A_\Gamma$ satisfy the relative uniform power alternative with uniform exponent $N$ for every $N\ge N_0$. Let us further assume that, for every complete subgraph $\Lambda\le\Gamma$, there exists $m_\Lambda\in\N_{\ge1}$ such that the group $A_\Lambda$ satisfies the power alternative with uniform exponent $m_{\Lambda}$. We denote by $K$ the number
    \[
    K=\operatorname{lcm}\left\{m_{\Lambda}:\text{$\Lambda\le\Gamma$ complete subgraph}\right\}
    \]
    Then $A_\Gamma$ satisfies the power alternative with uniform exponent
    \[
    N \coloneqq \min\{\alpha K:\text{$\alpha\in\N_{\ge1}$ and $\alpha K\ge N_0$}\}
    \]
\end{lem}

\begin{proof}
    Let $g,h\in A_{\Gamma}$. If $g$ and $h$ do not belong to a common parabolic subgroup based on a complete subgraph, then since $N\ge N_0$ by construction and $A_\Gamma$ satisfies the relative uniform power alternative with uniform exponent $N$, we have that $\Span{g^N,h^N}_{A_\Gamma}$ is either abelian or non-abelian free. If $g$ and $h$ belong to the conjugate of some parabolic $A_\Lambda$ for some complete subgraph $\Lambda\le\Gamma$, then $N$ is a multiple of $m_\Lambda$ by construction, and thus $g$ and $h$ also satisfy the power alternative with exponent $N$.
\end{proof}

As a particular case of this theorem, we obtain Corollary~\ref{corintro:uniform wpa}. We first need the following: 

\begin{thm}[\cite{ciobanu2020equations}*{Section~2}]
    \label{thm:dihedrals have wpa}
    Let $A_\Gamma$ be a spherical dihedral Artin group with label $m\in\N_{\ge3}$ and let us define
    \[
    m'=
    \begin{cases}
        m/2  &\text{if $m$ is even}\\
        2m  &\text{if $m$ is odd}
    \end{cases}
    \]
    Then $A_\Gamma$ contains a subgroups of index $m'$ that is isomorphic to $\Z\times F_k$ for some $k\in\N_{\ge2}$. In particular, $A_\Gamma$ satisfies the power alternative with uniform exponent $m'$.
\end{thm}

\begin{proof}[Proof of Corollary~\ref{corintro:uniform wpa}]
    By Theorem~\ref{thm:rel wpa artin}, $A_\Gamma$ satisfies the relative power alternative with uniform exponent $N$ for every $N\in\N_{\ge3}$. It is sufficient to show that every standard parabolic subgroup of $A_\Gamma$ based on a complete graph satisfies the uniform power alternative. Therefore, let $\Lambda\le\Gamma$ be a complete subgraph. Because $A_\Gamma$ is two-dimensional, $\Lambda$ has at most two vertices, the only non-trivial case being when~$\Lambda$ has exactly two vertices. In such case, Theorem~\ref{thm:dihedrals have wpa} yields the claim.
\end{proof}

\subsection{Further background on Artin groups}

We collect here some further background results that are not yet known for general Artin groups but which hold when restricting to the class of two-dimensional Artin groups. Note that triangle-free Artin groups are two-dimensional, in particular. In view of the results discussed in Section~\ref{sec:background on artin groups}, we have the following theorem.

\begin{thm}
    \label{thm:(2,2)-free triangle-free have good properties}
    If $A_\Gamma$ is $(2,2)$-free and triangle-free, then $A_\Gamma$ has the intersection property and the normaliser structure property.
\end{thm}

\boldparagraph{Normalisers of parabolic subgroups} As mentioned in Section~\ref{sec:wpa for Artin groups expressed as a visual splitting}, normalisers of parabolics of two-dimensional Artin groups are well-understood. For the class that we consider in this section, an even stronger result holds.

\begin{thm}[normalisers of parabolics, {\citelist{\cite{martin2022acylindrical}*{Lemma~4.5}\cite{godelle2007artin}*{Corollary~4.12}}}]
    \label{thm:normalisers parabolics two-dim}
    Let $A_\Gamma$ be a $(2,2)$-free triangle-free Artin group and let $\Lambda\le\Gamma$ be a full subgraph;
    \begin{enumerate}
        \item if~$\Lambda$ consists of a single vertex $a$, then  $\norm{A_\Gamma}{A_\Lambda}=\operatorname{C}_{A_\Gamma}(a)$ and, moreover, there exists $k\in\N$ such that $\operatorname{C}_{A_\Gamma}(a)\cong\Z\times F_k$;
        \item if $\abs{\Ver\Lambda}\ge2$, then $\norm{A_\Gamma}{A_\Lambda}=A_\Lambda$. 
    \end{enumerate}
\end{thm}

The way to conjugate parabolic subgroups amongst each other is also fully understood for this class of Artin groups, using the key concept of ribbons, which was introduced by Paris~\cite{paris1997parabolic}. The following theorem, which we will need for the proof of Lemma~\ref{lem:stab_segments_2D}, follows as a corollary of the facts that two-dimensional Artin groups satisfy the so-called ``ribbon conjecture''~\cite{godelle2007artin}*{Conjecture~1}.

\begin{thm}[conjugated  parabolics, \cite{godelle2007artin}*{Theorem~3}]
    \label{thm:ribbons}
    Let $A_\Gamma$ be a two-dimen\-sional Artin group and let $X,Y\subset\Ver\Gamma$ be subsets of cardinality at least two. If $A_X$ and $A_Y$ are conjugated in $A_\Gamma$, then $X=Y$.
\end{thm}

\boldparagraph{Closure under taking roots} The last property we want to address is the one of having the parabolic subgroups that are closed under taking roots.

\begin{defi}[closure under taking roots]
    Let $A_\Gamma$ be an Artin group; we say that the parabolic subgroups of $A_\Gamma$ are \emph{closed under taking roots} if the following property holds: for every parabolic subgroup $P\le A_\Gamma$, for every $g\in A_\Gamma$ and for every $n\in\N_{\ge1}$, if $g^n\in P$, then $g\in P$.
\end{defi}

Although this property is not known in general, it holds for the class of Artin groups we are studying in this section: 

\begin{thm}[\cite{godelle2007artin}*{Corollary~3.8}]\label{thm:parabolics_closed_roots}
Let $A_\Gamma$ be a two-dimensional Artin group. Then parabolic subgroups of $A_\Gamma$ are closed under taking roots.
\end{thm}

\begin{lem}
    \label{lem:fps and stable fps coincide}
    Let $A_\Gamma$ be an Artin group whose parabolics are closed under taking roots and let us assume that it is expressed as a visual splitting $A_{\Gamma_1}*_{A_{\Gamma_0}}A_{\Gamma_2}$ with associated Bass-Serre tree~$T$. For every element $g\in A_\Gamma$ acting elliptically on~$T$,
    \[
    \stabfix g = \fix g.
    \]
    In particular, if $g,h\in A_\Gamma$ are elements acting elliptically on~$T$, then either $\Span{g,h}\cong F_2$ or $g$ and $h$ belong to a common proper parabolic subgroup. 
\end{lem}

\begin{proof}
    For the first statement, note that stabilisers of fixed-point sets are parabolic subgroups and these are are closed under taking roots by hypothesis. The proof of the second statement goes along the lines of the proof of Lemma~\ref{lem:elliptic-elliptic general case}. 
\end{proof}

\subsection{Relative uniform power alternative for (2,2)-free triangle-free groups}

The aim of this section is to prove Theorem~\ref{thm:rel wpa artin}, namely that $(2,2)$-free triangle-free Artin groups satisfy the relative uniform power alternative. Therefore, let $A_\Gamma=A_{\Gamma_1}*_{A_{\Gamma_0}}A_{\Gamma_1}$ be a triangle-free and $(2,2)$-free Artin group expressed as a visual splitting and let~$T$ be the Bass-Serre tree associated to the splitting. We start with an elementary observation:

\begin{lem}[properties of fixed trees]
    Let $P\le A_\Gamma$ be a parabolic subgroup.
    \begin{enumerate}
        \item the fixed-point set $\fix P = \{x\in T: P\cdot x=x\}$ is a subtree of~$T$;
        \item for every $g\in A_\Gamma$, $g\cdot \fix P = \fix{gPg^{-1}}$;
    \end{enumerate}
\end{lem}

We call $\fix P$ the \emph{fixed tree} associated to the parabolic subgroup $P$.

\begin{lem}\label{lem:stab_segments_2D}
     Let $\gamma\subset T$ be a geodesic segment that contains at least two edges of~$T$. If $\pstab{A_\Gamma}{\gamma}$ is an infinite group, then it is an infinite cyclic parabolic subgroup, and~$\gamma$ is contained in a fixed tree $\fix a$, for some $a\in\Ver\Gamma$.
\end{lem}

\begin{proof}
    Because the parabolic intersection property holds for $A_\Gamma$, $\pstab{A_\Gamma}{\gamma}$ is a parabolic subgroup of~$A_\Gamma$ \cite{blufstein2022parabolic}*{Theorem~1.3}. Let us distinguish two cases, depending on the rank of $\pstab{A_\Gamma}{\gamma}$. \medskip
    
    If $\pstab{A_\Gamma}{\gamma}$ is cyclic, then it is of the form $g\Span{a}g^{-1}$ for some $g\in A_\Gamma$ and some $a\in\Ver{\Gamma_0}$. In this case,~$\gamma$ is contained in the fixed tree $g\cdot \fix a$. \medskip
    
    Let us assume by contradiction that $P=\pstab{A_\Gamma}{\gamma}$ is a parabolic subgroup of~$A_\Gamma$ of rank at least two. Let $e=\{A_{\Gamma_1},A_{\Gamma_2}\}$ be the canonical fundamental domain for the action $A_\Gamma\actson T$. Up to a translation in~$T$, we may assume that~$\gamma$ contains $e$. As $A_{\Gamma_2}$ acts transitively on the edges around the vertex of $e$ corresponding to the coset of $A_{\Gamma_2}$, let $h\in A_{\Gamma_2}\setminus A_{\Gamma_0}$ be such that~$\gamma$ contains the concatenation $e\cup h\cdot e$. Because~$P$ fixes $e$ pointwise, we obtain that $P\le\pstab{A_\Gamma}{e}=A_{\Gamma_0}$. Thus, by Theorem~\ref{thm:inclusion of parabolic subgroups}, $P$ is of the form $P=g^{-1}A_\Lambda g$ for some $g\in A_{\Gamma_0}$ and some $\Lambda\le\Gamma_0$. Because $P$ also fixes $h\cdot e$ pointwise, we obtain that $P\le\pstab{A_\Gamma}{h\cdot e}=hA_{\Gamma_0}h^{-1}$ or, equivalently, $h^{-1}Ph\le A_{\Gamma_0}$. Again, by Theorem~\ref{thm:inclusion of parabolic subgroups} there are $M\le\Gamma_0$ and $k\in A_{\Gamma_0}$ such that $h^{-1}Ph=kA_M k^{-1}$ or, equivalently, $(hk)A_M(hk)^{-1}=P$. In particular, the element $ghk$ conjugates $A_\Lambda$ and $A_M$ in $A_\Gamma$. Since $A_\Gamma$ is two-dimensional and~$\Lambda$ contains at least two vertices, it follows from Theorem~\ref{thm:ribbons} that $\Lambda = M$, hence $ghk$ normalises $A_\Lambda$. It now follows from Theorem~\ref{thm:normalisers parabolics two-dim} that $ghk$ belongs to $A_\Lambda$ itself, giving that $h\in A_{\Gamma_0}$, in particular, hence $h\cdot e =e$, a contradiction. 
\end{proof}

\begin{cor}
    If $A_\Lambda$ is a parabolic subgroup of $A_\Gamma$ on more than one vertex, then the fixed-point set of $A_\Lambda$ in~$T$ is either empty, a single vertex, or a single edge.
\end{cor}

Let us now study the different possible dynamics of pairs of elements of $A_\Gamma$ on the Bass-Serre tree~$T$.

\begin{lem}\label{lem:Uniform_hyp_hyp}
    Let $g, h\in A_\Gamma$ be two elements with $g$ acting loxodromically on~$T$ with translation length $\tau(g)$. If the intersection $\mathrm{Axis}(g)\cap \Min h$ has length at least $2\max\{\tau(g), \tau(h)\}+2$ (with the usual convention that $\tau(h)=0$ if $h$ acts elliptically), then $\mathrm{Axis}(g)\subset \Min h$.
\end{lem}

\begin{proof}
    Let $\Lambda \coloneqq \mathrm{Axis}(g)\cap \Min h$, and suppose that this overlap has length at least $2\max\{\tau(g), \tau(h)\}+2$. Let $\Lambda' \subset \Lambda$ be an initial subsegment of~$\Lambda$ of length~$2$. Up to replacing $g$ by $g^{-1}$, we can assume that $g$ and $h$ translate in the same direction along~$\Lambda$, and we get $hg\Lambda' = gh\Lambda'$. In particular, the commutator $[g,h]$  fixes pointwise $\Lambda'$. By Lemma~\ref{lem:stab_segments_2D}, there exists a conjugate $s$ of a standard generator, and an integer~$n$ such that $[g,h]=s^n$. By applying the height homomorphism $A_\Gamma\rightarrow \mathbb{Z}$ that maps every standard generator to $1$, we get that $n=0$, hence $g$ and $h$ commute. In particular, $g$ preserves $\Min h$. Since $g$ acts loxodromically by hypothesis, it follows that $\mathrm{Axis}(g) \subset\Min h$.
\end{proof}

\begin{cor}\label{cor:Uniform_hyp_hyp}
    Let $g, h\in A_\Gamma$ be two elements with $g$ acting loxodromically on~$T$. One of the following two conditions hold:
    \begin{enumerate}
        \item the elements $g$ and $h$ belong to a common proper parabolic subgroup;
        \item for every $n\ge3$, $\Span{g^n,h^n}_{A_\Gamma}$ is either abelian free or non-abelian free. 
    \end{enumerate}
\end{cor}

\begin{proof}
    Let $\Lambda \coloneqq \mathrm{Axis}(g)\cap \mathrm{Min}(h)$. We consider two cases, depending on the size of this overlap. In the case of $h$ acting loxodromically as well, it is not restrictive to assume that $\tau(g)\ge\tau(h)$.
    
    First assume that $\abs \Lambda < 2\max\{\tau(g),\tau(h)\}+2=2\tau(g)+2$. Since the translation length of $g$ on~$T$ is even, it follows that, for every $k\geq 3$,  the segments $g^k\Lambda$ and~$\Lambda$ are disjoint. It now follows from a standard ping-pong argument that $g^k$ and $h$ (and hence $g^k$ and $h^k$) generate a free subgroup. 

     If $\abs \Lambda \geq 2\tau(g)+2$, then it follows from Lemma~\ref{lem:Uniform_hyp_hyp} that $\mathrm{Axis}(g) \subset \Min h$. In particular, $g$ and $h$ normalise $P\coloneqq \pstab{A_\Gamma}{\Lambda}$. Since $A_\Gamma$ satisfies the Intersection Property, it follows that $P$ is a parabolic subgroup. There are thus two cases to consider: if $P$ is a cyclic parabolic subgroup, then its normaliser is of the form $\mathbb{Z}\times F_k$ by the first case of Theorem~\ref{thm:normalisers parabolics two-dim}, so $g$ and $h$ (hence their powers) either commute or generate a non-abelian free subgroup. If $P$ is a parabolic subgroup on at least two generators, then $P$ is self-normalising by Theorem~\ref{thm:normalisers parabolics two-dim}, point two. Thus, $g$ and $h$ are contained in $\norm{A_\Gamma}{P}=P$, which is a proper parabolic subgroup of $A_\Gamma$, as it is contained in some edge stabiliser.
\end{proof}

We are now ready to prove Theorem~\ref{thm:rel wpa artin}:

\begin{proof}[Proof of Theorem~\ref{thm:rel wpa artin}]
    We prove the result by induction on the number of generators. The initialisation case is the case of dihedral Artin groups and this follows from Theorem~\ref{thm:dihedrals have wpa}.

    Let us now assume that the result has been proved for graphs on less than $n\ge 2$ vertices, and let $A_\Gamma$ be an Artin group with a presentation graph~$\Gamma$ on $n+1\ge 3$ vertices, satisfying the assumption of the theorem.  Since $A_\Gamma$ is triangle-free on at least three vertices, it follows that~$\Gamma$ is not a complete graph, hence we can consider a visual splitting of $A_\Gamma$, and the  action of $A_\Gamma$ on the corresponding Bass--Serre tree~$T$. 
    Let $g, h$ be two elements of $A_\Gamma$. There are two cases to consider: 
    \begin{itemize}
        \item if at least one between $g$ and $h$ acts loxodromically on~$T$, it follows from Corollary~\ref{cor:Uniform_hyp_hyp} that either $g$ and $h$ belong to a common proper parabolic subgroup or, for every $k\ge3$, $\Span{g^k,h^k}_{A_\Gamma}\cong F_2$;
        \item if $g$ and $h$ both act elliptically, then the second part of Lemma~\ref{lem:fps and stable fps coincide} gives the claim.
    \end{itemize}
    In both cases, either for every $k\ge3$, $g^k$ and $h^k$ generate a non-abelian free subgroup, or $g$ and $h$ belong to a common proper parabolic subgroup on at most $n$ generators. The result now follows from the induction hypothesis.
\end{proof}

\subsection{Application: Uniform exponential growth} In this subsection, we use previous results on the uniform power alternative to obtain a strong form of uniform exponential growth for certain Artin groups:

\begin{thm}\label{thm:Artin_uniform_growth}
    Let $A_\Gamma$ be a triangle-free $(2,2)$-free Artin group that is not of spherical type. Let $m$ be the smallest multiple of $\mathrm{lcm}_{ab \in \Edge\Gamma}(m_{ab}')$ such that $m \geq 3$, where the coefficients $m_{ab}'$ were defined in Theorem~\ref{thm:dihedrals have wpa}.
    
    Let $S$ be a generating set of $A_\Gamma$. Then there exist $s, t \in S$ such that $s^m$ and $t^m$ generate a non-abelian free subgroup of $A_\Gamma$.
    In particular, $A_\Gamma$ has uniform exponential growth.
\end{thm}

The proof of this theorem relies on the action of these Artin groups on their Deligne complex. We recall the definition of these complexes, with both their simplicial and cubical structure.  

\begin{defi}[simplicial and cubical Deligne complex \cite{charney1995k}]
Let $A_\Gamma$ be an Artin group. The \emph{(simplicial) Deligne complex} is the following simplicial complex:
\begin{itemize}
    \item vertices of $D_\Gamma$ correspond to left cosets $gA_{\Gamma'}$ of spherical standard parabolic subgroups of $A_\Gamma$;
\item simplices correspond to chains $gA_{\Gamma_0} \subsetneq \cdots \subsetneq gA_{\Gamma_k}$, for $g \in A_\Gamma$ and $\Gamma_0 \subsetneq \cdots \subsetneq \Gamma_k \subset \Gamma$ such that each $A_{\Gamma_i}$ is spherical.
\end{itemize}
 The \emph{cubical Deligne  complex} $C_\Gamma$  is the cube complex with the same vertex set as $D_\Gamma$, and where cubes correspond to combinatorial intervals (for the inclusion) between vertices of $C_\Gamma$ of the form $gA_{\Gamma_0}$ and $gA_{\Gamma_k}$, whenever $g \in A_\Gamma$ and $\Gamma_0 \subsetneq \Gamma_k$.

 The group $A_\Gamma$ acts on $C_\Gamma$ and on $D_\Gamma$ by left multiplication on left cosets. These actions are without inversions: An element of $A_\Gamma$ stabilises a cube of $C_\Gamma$ (resp. a simplex of $D_\Gamma$) if and only if it fixes it pointwise.
\end{defi}

The geometry of Deligne complexes is well understood for two-dimensional Artin groups and for FC-type Artin groups, among other classes. Triangle-free Artin groups are in particular two-dimensional and FC-type Artin groups. Moreover, triangle-free Artin groups that are $(2, 2)$-free are also of hyperbolic type. We therefore have the following results:

\begin{thm}[\cite{charney1995k}*{Theorem 4.3.5}]
    Let $A_\Gamma$ be a triangle-free Artin group. Then the cubical Deligne complex $C_\Gamma$ is a \cat 0 cube complex. 
\end{thm}

\begin{thm}[\citelist{\cite{charney1995k}*{Proposition 4.4.5}\cite{martin2022acylindrical}*{Section 3.1}}]
    Let $A_\Gamma$ be a triangle-free Artin group. Then the simplicial Deligne complex $D_\Gamma$ admits a \cat 0 metric. Moreover, if $A_\Gamma$ is $(2,2)$-free, then $A_\Gamma$ is of hyperbolic type, and $D_\Gamma$ admits a \cat {-1} metric.
\end{thm}

Note that one can obtain the simplicial Deligne complex by subdividing in a suitable way the cubes of the corresponding cubical Deligne complex. Thus, the results about the action of $A_\Gamma$ on $D_\Gamma$, and in particular the description of the minsets and fixed-point sets in $D_\Gamma$ obtained in \cite{martin2022acylindrical}, carry over to the cubical setting without any change: 

\begin{lem}[\cite{martin2022acylindrical}]
    Let $A_\Gamma$ be a triangle-free Artin group. Let $g \in A_\Gamma$ be an element acting elliptically on the cubical Deligne complex $C_\Gamma$. 
    \begin{itemize}
        \item If $g$ is conjugated to a power of a standard generator, then the fixed-point set of $g$ is a tree contained in the $1$-skeleton of $C_\Gamma$, called a \emph{standard tree}. 
        \item Otherwise, the fixed-point set of $g$ is a single vertex of $C_\Gamma$, whose stabiliser is conjugated to a dihedral standard parabolic subgroup of $A_\Gamma$. 
    \end{itemize}
\end{lem}

\begin{lem} Let $A_\Gamma$ be a triangle-free Artin group. Let $g \in A_\Gamma$ be an element acting elliptically on $C_\Gamma$. Then  $\mathrm{Min}(g) = \mathrm{Min}(g^k)$ for every $k \geq 1$ and $g\in A_\Gamma$.
\end{lem}

\begin{proof}
    Stabilisers of points of $C_\Gamma$ are parabolic subgroups of $A_\Gamma$ by construction. Since parabolic subgroups of two-dimensional Artin groups are closed under taking roots by  Theorem~\ref{thm:parabolics_closed_roots}, the result follows immediately.
\end{proof}

\begin{proof}[Proof of Theorem~\ref{thm:Artin_uniform_growth}]
Let $s, t$ be two elements of $S$, and let us first assume that $s, t$ belong to a dihedral parabolic subgroup $H$ with label $m_{ab}$. Since $H$ contains an index $m_{ab}'$ subgroup that is of the form $\mathbb{Z}\times F_k$ for some $k \geq 1$, it follows that $s^{m_{ab}'}$ and  $t^{m_{ab}'}$ either commute or a generate a free subgroup. Otherwise, $s, t$ do not belong to a dihedral parabolic subgroup, and it follows from Theorem~\ref{thm:rel wpa artin} that $s^3, t^3$ either commute or generate a free subgroup. In any case, it follows that $s^m$ and $t^m$ either commute or generate a free subgroup. 
If for one pair of elements $s, t \in S$, we have that $s^m$ and $t^m$ generate a free subgroup, we are done. So, let us assume by contradiction that for every pair $s, t$ of distinct elements of $S$, the elements $s^m$ and $t^m$ commute. 

First, suppose that some element $s_0\in S$ acts loxodromically on $C_\Gamma$, hence on $D_\Gamma$, which we think of as endowed with its \cat{-1} metric. In particular, $s_0$ admits a unique axis $L\subset D_\Gamma$. Let $s\in S$ with $s\neq s_0$. Since $s^m$ and $s_0^m$ commute, it follows that either $s$ is loxodromic with the same axis $L$ as $s_0$, or $s$ is elliptic with fixed-point set $\mathrm{Min}(s) = \mathrm{Min}(s^m)$ a standard tree containing $L$. Note in particular that there exists a unique standard tree of $D_\Gamma$ containing $L$  by \cite{hagen2024extra}*{Corollary 2.18} (Note that this result, and the others from \cite{hagen2024extra} used in this proof, are stated for large-type Artin groups, but the reader can check that the proof holds more generally for two-dimensional Artin groups as their proof only use the fact that the Deligne complex admits a \cat 0 metric). Thus, either all elements of $S$ act loxodromically on $D_\Gamma$ with axis $L$, or there exists a unique standard tree $T$ such that the minset of every element of $S$ is contained in $T$. In the former case, we get that every $s\in S$, hence $\langle S \rangle = A_\Gamma$, is contained in  $\stab {A_\Gamma}{L}$. As the pointwise stabiliser of $L$ is either infinite cyclic (if it is contained in a standard tree) or trivial (otherwise) by \cite{hagen2024extra}*{Corollary 2.17}, it follows that $A_\Gamma$ is virtually abelian, a contradiction. In the latter case, it follows from \cite{hagen2024extra}*{Lemma 2.15} that every element of $S$, and hence $\langle S \rangle = A_\Gamma$, is contained in $\stab {A_\Gamma}{T}$, which is of the form $\mathbb{Z}\times F_k$ by \cite{martin2022acylindrical}*{Lemma 4.5}. Since $A_\Gamma$ is not abelian, it follows that some pair of elements of $S$ generates a non-abelian free subgroup, contradicting our assumption.

Now suppose that all elements $s\in S$ act elliptically on $C_\Gamma$. This  implies that for every pair, the minsets $\mathrm{Min}(s^m) = \mathrm{Min}(s)$ and $\mathrm{Min}(t^m)=\mathrm{Min}(t)$ have a non-empty intersection, otherwise some large powers would generate a free subgroup by \cite{martin2023tits}*{Proposition C}. 
Since the fixed-point sets $\mathrm{Min}(s)$ are convex subcomplexes of $C_\Gamma$ and pairwise intersect, the Helly property for CAT(0) cube complexes \cite{roller98habilitation}*{Theorem 2.2} implies that there exists a vertex of $D_\Gamma$ fixed by all $s\in S$, hence by $\langle S \rangle= A_\Gamma$. Thus, $A_\Gamma$ is a dihedral Artin group, a contradiction.
\end{proof}

\section{Examples, additional applications, and questions}
\label{sec:recovering}

We now define a class of groups, starting with hyperbolic groups and closed under several operations, such that all groups in the class satisfy the power alternative (uniformly).  The goal is to recover several existing examples in a unified way; one could probably define such a class of groups more generally, in terms of acylindrical actions on hyperbolic complexes such that the stabilisers of simplices are in the class and satisfy suitable ``rotation'' conditions along the lines of \cite{DahmaniGuirardelOsin} or \cite{BHMS}.  However, in the interest of concreteness, we content ourselves with a less general approach sufficient for our favourite examples.

We refer the reader to, for instance, \cite{Bowditch:relhyp} for the definition of \emph{relative hyperbolicity} of a pair $(G,\mathcal P)$, where $\mathcal P$ is a set of subgroups of the finitely generated group $G$.  We also recall:

\begin{defi}[Acylindricity]\label{defn:acylindricity}
Given a metric space $X$ and a group $G$, an action $G\to \operatorname{Isom}(X)$ is \emph{acylindrical} if for each $\epsilon\geq 0$ there exist $R(\epsilon),N(\epsilon)<\infty$ such that $$|\{g\in G:d_X(x,gx)\leq\epsilon\ \text{and } d_X(y,gy)\leq\epsilon\}|\leq N(\epsilon)$$ for all $x,y\in X$ with $d_X(x,y)>R(\epsilon)$.  The functions $R$ and $N$ are the \emph{acylindricity parameters}. 
\end{defi}

We are usually interested in the case where $X$ is a hyperbolic geodesic space; in this setting the notion was defined in \cite{bowditch2008tight}.  In the case where $X$ is a simplicial tree, the notion of acylindricity of the $G$--action was formulated earlier, and equivalently, by Sela in \cite{sela1997acylindrical}: there exist $K,c\in\integers_{\geq0}$ such that if $v,w\in \Ver X$ satisfy $d_X(v,w)>K$, then $|\stab{G}{v}\cap \stab{G}{w}|\leq c$.

\begin{rem}\label{rem:power alternative function}
To facilitate quantitative statements about the power alternative, we introduce some notation.  Let $G$ be a group.  The \emph{power alternative function} $\paf_G:G^2\to \N_{\ge1}\cup\{\infty\}$ is defined as follows.  Given $g,h\in G$, let $\paf_G(g,h)$ be the smallest positive integer $n$ such that either $[g^n,h^n]=1$ or $\Span{g^n,h^n}\cong F_2$.  If no such $n$ exists, then $\paf_G(g,h)=\infty$.  Hence $G$ satisfies the power alternative if and only if $\paf_G(g,h)<\infty$ for all $g,h\in G$.  We also let $\paf(G)=\sup_{g,h\in G}\paf_G(g,h)$.  So, $G$ satisfies the uniform power alternative if $\paf(G)<\infty$.  Finally, if $\mathcal P$ is a family of groups, we let $\paf(\mathcal P)=\sup_{P\in\mathcal P}\paf(P)$.
\end{rem}

\begin{defi}[the classes $\paclass$ and $\upaclass$]\label{defn:WPA-class}
We define classes $\paclass$ and $\upaclass$ of groups as follows.  First, let $\upaclass_0$ be the class of finitely generated groups $G$ that fit into an exact sequence $Z\hookrightarrow G\twoheadrightarrow H$ where $H$ is a hyperbolic group and $Z$ is free abelian of rank at most $1$ that is central in $G$.  Define $\upaclass$ to be the smallest class of finitely generated groups that satisfies all of the following closure properties:
\begin{enumerate}
    \item $\upaclass_0\subset\upaclass$.
    \item If $G\in\upaclass$ and $[G':G]<\infty$ then $G'\in\upaclass$.
    \item If $G,G'\in\upaclass$ then $G\times G'\in\upaclass$.
    \item Let $G$ admit a cocompact  acylindrical action on a tree $T$ such that all vertex stabilisers belong to $\upaclass$ and are closed under roots, i.e. for $g\in G$ and $v\in T$, if $g^nv=v$ for some $n\neq 0$, then $gv=v$. Then $G\in\upaclass$.\label{item:tree-comb}
    \item Let $\Gamma$ be a finite simplicial graph, for each $v\in\Ver{\Gamma}$ let $G_v$ be a nontrivial group, and let $G$ be the graph product of $\{G_v:v\in \Gamma\}$.  If each $G_v\in \upaclass$, then $G\in\upaclass$.
    \item If $G$ is hyperbolic relative to a finite collection $\mathcal P$ of subgroups and each $H\in\mathcal P$ belongs to $\upaclass$, then $G\in\upaclass$.
\end{enumerate}
The larger class $\paclass$ is defined identically, replacing $\upaclass$ with $\paclass$, except condition \eqref{item:tree-comb} changes to:
\begin{enumerate}[(i)]
\setcounter{enumi}{3}
    \item $G$ acts on a tree $T$ with the \goodname\ property, and all stabilisers of points in $\Ver{T}\cup\partial T$ belong to $\paclass$.\label{item:tree-comb-prime}
\end{enumerate}
\end{defi}

The point of Definition \ref{defn:WPA-class} is just to enable a succinct statement of the following:

\begin{thm}[omnibus combination theorem]\label{thm:meta-pa}
If $G\in\paclass$, then $G$ satisfies the power alternative.
If $G\in\upaclass$, then $G$ satisfies the uniform power alternative.
\end{thm}

\begin{proof}
Theorem \ref{thm:hyperbolic} and Lemma \ref{lem:central-extension} below show that $G\in \upaclass_0$ implies $\paf(G)<\infty$.  Lemma \ref{lem:wpa virtual property} and Lemma \ref{lem:wpa stable direct product} show that satisfying the [uniform] power alternative is stable under passing to finite index supergroups and taking direct products.  Theorem \ref{thmintro:WPA for actions on trees} implies that groups $G$ as in condition \eqref{item:tree-comb-prime} satisfy the non-uniform power alternative.

If $G$ is hyperbolic relative to a finite collection of groups $H$ such that $\paf(H)<\infty$, then $\paf(G)<\infty$ by Theorem \ref{thm:rel-hyp}.  In order to apply the theorem, we need to know that there is a bound on the orders of torsion elements of $G$, but this property holds for hyperbolic groups and persists under all of the operations in Definition \ref{defn:WPA-class}.

Now, suppose $G$ acts acylindrically and cocompactly on a tree $T$,  acylindrically and cocompactly, and the vertex stabilisers $G_v$ are all closed under roots and satisfy $\paf(G_v)<\infty$.  Then Proposition \ref{prop:acyl-tree-action} below implies $\paf(G)<\infty$ (with the corresponding observation about torsion).

The statement about graph products follows from \cite{antolin2015tits}*{Corollary 1.5}.  This concludes the proof.
\end{proof}

It remains to prove the results from the proof of Theorem \ref{thm:meta-pa}, which we do in the next few subsections.

\subsection{Central extensions of hyperbolic groups}\label{subsec:central-extensions}

The first result is well-known and can be recovered from arguments in \cite{Gromov}, but we will recover it as a special case of a more general statement:

\begin{thm}
    \label{thm:hyperbolic}
    Let $G$ be a group equipped with a finite generating set such that the resulting word metric on $G$ is $\delta$--hyperbolic.  Then there exists a natural number $n(G)$ such that $\paf(G)\leq n(G)$.  More precisely, there exists a positive integer $n=n(G)$ such that for all $g,h\in G$ of infinite order:
    \begin{enumerate}
        \item if $g$ and $h$ share a limit point in $\partial G$, then both their limit points coincide and $\Span{g,h}_G$ contains a finite index $\Z$ subgroup containing $g^n$ and $h^n$, and otherwise,
        \item if $g$ and $h$ do not share any points in $\partial G$, then $\Span{g^n,h^n}_G\cong F_2$.
    \end{enumerate}
\end{thm}

\begin{proof}
This is the special case of Theorem \ref{thm:rel-hyp} where $\mathcal P=\{\{1\}\}$.
\end{proof}

\begin{lem}\label{lem:central-extension}
Let $H$ be a hyperbolic group and let 
$$1\to\Z\to G\to H\to 1$$
be an extension.  Then $G$ satisfies the uniform power alternative.
\end{lem}

\begin{proof}
Let $g,h\in G$ and let $\bar g,\bar h$ denote their images in $H$.  Choose $n\geq 1$ such that $\langle \bar g^n,\bar h^n\rangle_H$ is either nonabelian free or cyclic, using Theorem \ref{thm:hyperbolic}.  In the former case, $\langle g^n,h^n\rangle_G$ is again free; in the latter, $\langle g^{2n}, h^{2n}\rangle_G$ abelian since the preimage of $\langle \bar g^{2n},\bar h^{2n}\rangle_H$ in $G$ is $\integers^2$ or $\integers$.
\end{proof}

\subsection{Actions on hyperbolic spaces}
In this section, we prove the following two results, whose proofs we postpone until after some lemmas.

\begin{notation}
    In order to lighten the notations, in the rest of this section we will simply write as $G_v$ the stabiliser $\stab G v$ of a vertex $v$ of $T$.
\end{notation}

\begin{prop}\label{prop:acyl-tree-action}
Let $G$ act acylindrically and cocompactly on a tree $T$.  Suppose that $\paf(G_v)<\infty$ and $G_v$ is closed under roots, for each $v\in\Ver T$. Then $\paf(G)<\infty$ and $G$ has a bound on the orders of torsion elements.
\end{prop}

The power alternative for $G$ acting acylindrically on a tree is Corollary \ref{cor:acyl}, but below, we will argue slightly differently to get the uniform version.

In the next theorem, we consider a relatively hyperbolic pair $(G,\mathcal P)$.  Given a finite generating set $S$ of $G$, the \emph{coned-off Cayley graph} $\cay{G,S;\mathcal P}$ is obtained from the Cayley graph $\cay{G,S}$ by coning off each coset $gP,\ P\in\mathcal P,g\in G$; see \cite{Bowditch:relhyp} for details.

\begin{thm}\label{thm:rel-hyp}
Let $G$ be a finitely generated group that is hyperbolic relative to a finite collection $\mathcal P$ of subgroups.  Let $\delta\geq 0$ be such that $G$ admits a finite generating set $S$ such that $\cay{G,S;\mathcal P}$ is $\delta$--hyperbolic. 

Then $G$ satisfies the power alternative provided each $H\in\mathcal P$ does.  Moreover, suppose that there is a bound $B<\infty$ on the orders of finite-order elements of subgroups in $\mathcal P$ and $\paf(\mathcal P)<\infty$. 
Then $\paf(G)$ is bounded in terms of $\paf(\mathcal P),B,\delta$, and the acylindricity parameters of the $G$--action on $\cay{G,S;\mathcal P}$. 

In particular, if $G$ has bounded torsion, then $G$ satisfies the power alternative uniformly if $\mathcal P$ does.
\end{thm}

We now analyse certain two-generated subgroups of an acylindrically hyperbolic group, in order to unify the proofs of the above two statements.  The following is similar to arguments in, for instance, \cite{AbbottDahmani:P-Naive}.

Fix a group $G$ and a $\delta$--hyperbolic graph $X$, whose graph metric we call $d_X$.  Suppose that $G$ acts by acylindrically by isometries on $X$, with $R:\reals_{\geq 0}\to \reals_{\geq0}$ and $N:\reals_{\geq 0}\to\integers_{\geq0}$  as in Definition \ref{defn:acylindricity}.

Given such an action, we say that a quantity is \textit{uniform} if it depends on $\delta$ and the functions $(R,N)$ but not on any particular points in $X$ or elements of $G$.

Given $g\in G$, let 
$$A(g)=\{x\in X:d_X(x,gx)\leq \inf_{y\in X}d_X(y,gy)+10\delta\}.$$

Acylindricity provides a uniform $n_1<\infty$ such that, for all $g\in G$, one of the following holds:
\begin{itemize}
    \item $g$ is loxodromic on $X$.  In this case, $A(g)$ is $n_1$--quasiconvex and $(n_1,n_1)$--quasiisometric, in the subspace metric, to a line.  Moreover, the unique maximal elementary subgroup $E(g)\leq G$ containing $g$ is the stabiliser of $A(g)$. (See \cite[Lemma 6.5]{DahmaniGuirardelOsin}.)

    \item $\langle g\rangle$ has bounded orbits in $X$.  In this case, there exists $x\in X$ such that $\diam_X(\langle g\rangle \cdot x)\leq 5\delta$ (see, for instance, \cite[Lemma III.$\Gamma$.3.3]{bridson2013metric}).  So either $g$ has order at most $N(15\delta)$ or $\diam_X(A(g))\leq R(15\delta)$, and we can assume, by uniformly enlarging $n_1$, that $R(15\delta)\leq n_1$ and that $A(g)$ is $n_1$--quasiconvex.  
\end{itemize}

Now fix $g\in G$.  Observe:

\begin{lem}\label{lem:acyl-order-bound}
If $g$ has order at most $N(15\delta)$, then $\paf_G(g,h)$ is bounded uniformly, for any $h\in G$.
\end{lem}

So assume that $g$ has order more than $N(15\delta)$, and hence $A(g)$ has diameter at most $n_1$ when $g$ has bounded orbits in $X$.  For $g\in G$ loxodromic on $X$, fix $x_0\in X$ and let $\tau(g)=\lim_{n\to\infty}d_X(x_0,g^nx_0)/n$.   Since the action is acylindrical, there is a uniform $\tau_0>0$ such that $\tau(g)\geq \tau_0$ whenever $g$ is loxodromic, as shown in \cite{bowditch2008tight}.

\begin{defi}\label{defn:acyl-projection}
Given $g\in G$, let $\pi_g:X\to 2^{A(g)}$ be the coarse closest-point projection, i.e. for $x\in X$, let 
$$\pi_g(x)=\{y\in A(g):d_X(x,y)\leq d(x,A(g))+1\}.$$
\end{defi}

Since $X$ is $\delta$--hyperbolic and $A(g)$ is $n_1$--quasiconvex, there is a uniform constant $C$ with the following properties that together say $\pi_g$ is a uniformly coarsely lipschitz coarse retraction:
\begin{itemize}
    \item $\diam_X(\pi_g(x))\leq C$ for all $x\in X$.

    \item $d_X(\pi_g(x),\pi_g(y))\leq Cd_X(x,y)+C$ for all $x,y\in X$.

    \item $d_X(a,\pi_g(a))\leq C$ for all $a\in A(g)$.
\end{itemize}
Moreover, if $h\in G$, note that $A(hgh^{-1})=hA(g)$, and for all $x\in X$ we have $h\cdot \pi_g(x)=\pi_{hgh^{-1}}(h\cdot x)$.  (These are standard facts about coarse projection to quasiconvex subspaces of a hyperbolic space.)

\begin{lem}\label{lem:acyl-proj-bound}
There exists a uniform constant $C_1$ such that the following holds.  Let $g,h\in G$ have order at least $N(15\delta)$ and let $\tau=\min\{\tau(g),\tau(h)\}$.  Then all of the following hold:
\begin{enumerate}
    \item $\diam_X(\pi_h(A(ghg^{-1})))=\diam_X(\pi_h(gA(h)))\leq C_1\tau(h) + C_1$ unless $d_{\textup{Haus}}(A(h),gA(h))\leq C_1.$
    \item $\diam_X(\pi_h(A(g)))\leq C_1\tau+C_1$ unless $d_{\textup{Haus}}(A(g),A(h))\leq C_1.$
    \item If $d_X(A(h),gA(h))>C_1$, then $\diam_X(\pi_h(gA(h)))\leq C_1$, and the same holds replacing $gA(h)$ with $A(g)$.
\end{enumerate}
\end{lem}

\begin{proof}
The first assertion follows from the second, by replacing $g$ with $ghg^{-1}$.  Our assumption on the order implies the second assertion immediately if $g$ or $h$ has finite order, so we can assume they have infinite order.  Using uniform quasiconvexity of the uniform quasi-axes $A(g),A(h)$, acylindricity, and a thin quadrilateral argument, we obtain the third assertion, along with the conclusion that $\diam_X(\pi_h(A(g)))\leq C_1\tau+C_1$ unless $A(g)$ and $A(h)$ are at finite Hausdorff distance, as required.  (See also \cite[Section 6]{DahmaniGuirardelOsin}.)
\end{proof}

Lemma \ref{lem:acyl-proj-bound} and a standard thin quadrilateral argument give:

\begin{lem}[``Behrstock inequality'']\label{lem:acyl-behrstock}
There exists a uniform constant $C_2$ such that the following holds.  Let $h\in G$ be loxodromic and let $x\in X$.  Then the following hold for all $g\in G$:
\begin{itemize}
    \item Suppose that $E(h)\neq E(ghg^{-1})$.  Then $d_X(\pi_h(x),\pi_h(gA(h)))>C_2$ implies $$d_X(\pi_{ghg^{-1}}(x),\pi_{ghg^{-1}}(A(h)))\leq C_2.$$

    \item Suppose that $E(g)\neq E(h)$.  Then $d_X(\pi_h(x),\pi_h(A(g)))>C_2$ implies $d_X(\pi_g(x),\pi_g(A(h)))\leq C_2$.
\end{itemize}
\end{lem}

Now we study the subgroup of $G$ generated by uniform powers of $g,h\in G$ with $h$ loxodromic.

\begin{lem}[acylindrical, loxodromic-loxodromic]\label{lem:loxo-loxodromic}
There exists a uniform $p\in\integers_{>0}$ such that for all $g,h\in G$ that are both loxodromic on $X$, either $\langle g^p,h^p\rangle\cong F_2$ or $g^p,h^p$ commute.
\end{lem}

\begin{proof}
Let $\tau=\min\{\tau(g),\tau(h)\}$.  If $E(g)=E(h)$, then \cite[Lemma 6.8]{Osin:acyl} implies that $g$ and $h$ have uniform nonzero powers that commute and we are done.  By Lemma \ref{lem:acyl-proj-bound}, we can therefore assume $\diam_X(\pi_h(A(g)))\leq C_1\tau+C_1$. 

Let $L=C_2+C$, which is uniform.  For $a\in\{g,h\}$, let $$S_a=\{x\in X:d_X(\pi_a(x),\pi_a(A(b)))>L\},$$
where $b\in\{g,h\}-\{a\}$.  Since $A(a)$ is unbounded, $S_a\neq\emptyset$.  If $x\in S_a$, then by Lemma \ref{lem:acyl-behrstock}, $d_X(\pi_b(x),\pi_b(A(a)))\leq C_2$, so  $x\not\in S_b$.  We have shown $S_g\cap S_h=\emptyset$.

Since $\diam_X(\pi_a(A(b)))\leq C_1\tau(a)+C_1$, there is a uniform $p>0$ such that for all $n\in\integers-\{0\}$,  $$d_X(a^{np}\pi_a(A(b)),\pi_a(A(b)))>10L.$$  Indeed, recall that $\tau(a)\geq \tau_0$.  If $y\in \pi_a(A(b))$, then for any $k\in\integers-\{0\}$, we have $d_X(y,a^ky)\geq |k|\tau(a)/r_0$, where $r_0$ is uniform, since $y$ lies in the uniform quasi-axis $A(a)$.  For all $k\in\integers$ satisfying $|k|>10r_0L/\tau_0 +2C_1r_0(1+1/\tau_0)$, we therefore have 
$$d_X(a^ky,y)> \frac{10L\tau(a)}{\tau_0}+2C_1\tau(a)\left(1+\frac{1}{\tau_0}\right)\geq 10L+2C_1\tau(a)+2C_1,$$
using that $\tau(a)\geq \tau_0$.  Hence $$d_X(a^k\pi_a(A(b)),\pi_a(A(b)))> 10L+2C_1\tau(a)+2C_1-2\diam_X(\pi_a(A(b)))\geq 10L,$$
where the last estimate uses Lemma \ref{lem:acyl-proj-bound}.  So, $d_X(a^k\pi_a(A(b)),\pi_a(A(b)))>10L$, as required.   

We now argue that $h^{np}S_g\subsetneq S_h$ for all $n\neq 0$.  If $x\in S_g$, then by Lemma~\ref{lem:acyl-behrstock}, $$d_X(\pi_h(x),\pi_h(A(g)))\leq C_2,$$ so $$d_X(h^{np}\pi_h(x),\pi_h(A(g)))> 10L-(C_2+C)\geq9L,$$ so since $h^{np}\pi_h(x)=\pi_h(h^{np}x)$, we have $h^{np}x\in S_h$.  This shows $h^{np}S_g\subseteq S_h$.  Since $A(h)$ contains points that are at distance at most, say, $5L+C$ from $\pi_h(A(g))$, we in fact have $h^{np}S_g\subsetneq S_h$ for all $n\neq 0$.  A symmetric argument shows that $g^{np}S_h\subsetneq S_g$ for $n\neq 0$.  So, we have produced nonempty disjoint sets $S_g,S_h$ such that $g^{np}S_h\subsetneq S_g$ and $h^{np}S_g\subsetneq S_h$ for all $n\in\integers-\{0\}$, so by the ping-pong lemma, $\langle g^p,h^p\rangle\cong F_2$, and we are done.
\end{proof}

\begin{lem}[acylindrical, elliptic-loxodromic]\label{lem:elliptic-loxodromic}
There exists uniform $p\in\integers_{>0}$ such that the following holds.  Suppose that $g\in G$ is elliptic on $X$ and $h\in G$ is loxodromic.  Then either $\langle g^p,h^p\rangle\cong \langle g^p\rangle *\langle h^p\rangle$ or $\langle g\rangle \cap E(h)\neq \{1\}$.  
\end{lem}

\begin{proof}
If $g$ has order at most $N(15\delta)$, then we are done, so suppose not.  Then $A(g)$ has uniformly bounded diameter, so up to uniformly enlarging $C_1$, we can assume $A(g)$ and $\pi_h(A(g))$ both have diameter at most~$C_1$.  Also assume $g^k\not\in E(h)$ for all $k\neq 0$, so that by Lemma \ref{lem:acyl-proj-bound}, $\pi_h(g^kA(h))$ has diameter at most $C_1\tau(h)+C_1$ for all $k\ne0$. 

Let $L\geq C_2+C$ be a constant to be determined.  Let $$S_h=\{x\in X:d_X(\pi_h(x),\pi_h(A(g)))>L\}.$$  Let $S_g=X-S_h$.  

Exactly as in the proof of Lemma \ref{lem:loxo-loxodromic}, as long as $L/\tau(h)$ is bounded above by a uniform constant, there is a uniform $p>0$ such that $h^{np}S_g\subsetneq S_h$ for all $n\in\integers\setminus\{0\}$.

We now show that $L$ can be chosen as above in such a way that the resulting uniform $p$ can be chosen with the additional property that $g^{np}S_h\subsetneq S_g$ for all $n\in\integers\setminus\{0\}$.  

In the remainder of the argument, we will impose additional uniform conditions on $L$ as we go.  Fix $k\in\integers$ and consider the uniformly bounded sets $\pi_h(A(g))$ and $g^k\pi_h(A(g))=\pi_{g^khg^{-k}}(g^kA(g))=\pi_{g^khg^{-k}}(A(g))$. 
 These are respectively $\kappa$--close to $\pi_h(g^kA(h))$ and $\pi_{g^khg^{-k}}(A(h))$, where $\kappa$ is uniform, using Lemma \ref{lem:acyl-behrstock}.  

Suppose $x\in X$ satisfies $d_X(\pi_h(x),\pi_h(A(g)))>L$, so that $d_X(\pi_h(x),\pi_h(g^kA(h)))>L-\kappa$.  Then $d_X(\pi_{g^khg^{-k}}(g^kx),\pi_{g^khg^{-k}}(A(h)))>L-\kappa$.  So by Lemma \ref{lem:acyl-behrstock}, $d_X(\pi_h(g^kx),\pi_h(g^kA(h)))\leq C_2$, and hence $d_X(\pi_h(g^kx),\pi_h(A(g)))\leq C_2 + \kappa + C_1\tau(h)+C_1$.  So, letting $L=10(C+C_1+C_2+\kappa + C_1\tau(h))$, we have that $L/\tau(h)$ is uniform, and $x\in S_h$ implies that $g^kx\in S_g$ for $k\neq 0$.  Moreover, note that $A(g)\subset S_g$, but if $x\in S_h$, then $g^kx\not \in A(g)$, since otherwise $g^{-k}(g^kx)=x\in A(g)$, contradicting $x\in S_h$.  So $A(g)\cap g^kS_h=\emptyset$, whence $g^kS_h\subsetneq S_g$.  Hence the ping-pong lemma implies $\langle g^p,h^p\rangle\cong \langle g^p\rangle *\langle h^p\rangle$.
\end{proof}

\begin{rem}
Arguments  similar to the ones in the preceding lemmas also occur in, for instance, \cite[Proposition 2.1]{AbbottDahmani:P-Naive} and \cite[Lemma 2.3]{HagenSisto:separable}.
\end{rem}

Now we consider the case of a relatively hyperbolic pair $(G,\mathcal P)$.

\begin{proof}[Proof of Theorem \ref{thm:rel-hyp}]
Let $\widehat G=\cay{G,S;\mathcal P}$ be a coned-off Cayley graph and choose $\delta$ so that $\widehat G$ is $\delta$--hyperbolic.  We let $d$ be the word-metric on $G$ associated to $S$ and let $\hat d$ be the graph metric on $\widehat G$.  The natural action of $G$ on $\widehat G$ is acylindrical \cite[Theorem 5.4]{Osin:acyl}.  Now apply Lemma \ref{lem:loxo-loxodromic} and Lemma \ref{lem:elliptic-loxodromic} to obtain a uniform $p$ such that if $g,h\in G$ and at least one is loxodromic, then one of the following holds: $\langle g^p,h^p\rangle\cong \langle g^p\rangle*\langle h^p\rangle$, or $[g^p,h^p]=1$, or $h$ is loxodromic, $g$ is elliptic, and $E(h)\cap \langle g\rangle\neq\{1\}$.  

In either of the first two cases, we are done, so assume the third holds.  For the non-uniform conclusion, we can always assume $g,h$ have infinite order, so we are done since $E(h)$ is virtually cyclic.  For the uniform case, we argue as follows.  First, our hypotheses imply that $g$ having finite order implies that it has bounded order, so we can assume $g$ has infinite order, is elliptic, and $g^n\in E(h)$ for some $n\neq 0$.  Then by \cite[Lemma 6.8]{Osin:acyl}, $n$ can be chosen (non-uniformly) so that $g^n\in\langle h\rangle$, which contradicts ellipticity of $g$ (since $h$ is loxodromic), so we are done.

Hence it remains to consider the case where $\langle g\rangle$ and $\langle h\rangle$ both have bounded orbits on $\widehat G$.  We will use a projection argument analogous to the one in the proof of Lemma \ref{lem:loxo-loxodromic}, except with peripheral cosets replacing axes. For each $P\in\mathcal P$ and $a\in G$, let $\pi_{aP}:G\to aP$ be the coarse closest-point projection from \cite[Section 1.1]{Sisto:relhyp}.  Note that $\pi_{aP}(ax)=a\pi_P(x)$ for all $a,x\in G$.  

Suppose that $a,b\in G$ and $P,Q\in\mathcal P$ satisfy $aP\neq bQ$.  Then there exists $C$, depending only on $(G,d)$ and $\mathcal P$ such that $\pi_{aP_g}(bP_h)$ and $\pi_{bP_h}(aP_g)$ have $d$--diameter bounded by $C$; this follows from \cite[Lemma 1.9, Lemma 1.10, Lemma 1.15]{Sisto:relhyp}.  Also, up to uniformly enlarging $C$, \cite[Lemma 1.13, Lemma 1.15]{Sisto:relhyp} implies: if $x\in G$ satisfies $d(\pi_{bQ}(x),\pi_{bQ}(aP))>C$, then $d(\pi_{aP}(bQ),\pi_{aP}(x))\leq C$.  Let $L=100C$.

Now fix $g,h\in G$ having bounded orbits in $\widehat G$.  By replacing $g$ and $h$ by uniform positive powers, we can assume that there exist $P_g,P_h\in\mathcal P$ such that $g\in aP_ga^{-1}$ and $h\in bP_hb^{-1}$ for some $a,b\in G$.  Since $\mathcal P$ is almost-malnormal \cite{Bowditch:relhyp}, we can assume $aP_g$ and $bP_h$ are unique.  By the hypothesis about the power alternative in $\mathcal P$, we can assume $aP_g\neq bP_h$ (i.e., if $g,h$ belong to a common peripheral subgroup, then whichever of the uniform or non-uniform power alternatives is assumed for peripherals can be applied to the given $g,h$).  We can also assume that $g$ and $h$ have infinite order since, in the case where we are concerned with the uniform power alternative, we have assumed a uniform bound on the orders of torsion elements.

Since $\pi_{aP_g}(bP_h)$ has diameter at most $C$, there is a [uniform] $p$ such that $$d(g^{kp}\pi_{aP_g}(bP_h),\pi_{aP_g}(bP_h))>100L$$ for $k\in\integers-\{0\}$, and the same holds reversing the roles of $aP_g$ and $bP_h$ and replacing $g^p$ with $h^p$.  

Let $X_g=\pi_{aP_g}^{-1}(aP_g-N_L^G(\pi_{aP_g}(bP_h)))$ and let $X_h=\pi_{bP_h}^{-1}(bP_h-N^G_L(\pi_{bP_h}(aP_g)))$. If $x\in X_h$, then $d(\pi_{bP_h}(x),\pi_{bP_h}(aP_g))>L$, so $d(\pi_{aP_g}(x),\pi_{aP_g}(bP_h))\leq C$.  Hence $X_g\cap X_h=\emptyset$.  Moreover, for $k\neq 0$, we have $d(g^{pk}\pi_{aP_g}(x),\pi_{aP_g}(bP_h))>100L-2C>L$, so $g^{kp}x\in X_g$.  In fact, we have shown $g^{kp}X_h\subsetneq X_g$ for $k\neq 0$.  Similarly, $h^{kp}X_g\subsetneq X_h$ for $k\neq 0$.  The ping-pong lemma then implies $\langle g^p,h^p\rangle \cong F_2$.
\end{proof}

Now we turn to  actions on trees.

\begin{proof}[Proof of Proposition \ref{prop:acyl-tree-action}]
Let $G$ act acylindrically and cocompactly on the tree $T$.  Suppose that vertex stabilisers are closed under roots.  This has the following consequence: let $g\in G$ have finite order.  Then $g$ is elliptic, so let $v$ be a vertex with $gv=v$.  Since $g$ has finite order, there exists $n\neq 0$ such that $g^n$ acts on $T$ trivially, so since vertex groups are closed under roots, $g$ acts on $T$ trivially and hence belongs to the unique maximal finite normal subgroup of $G$; see \cite[Theorem 2.23]{DahmaniGuirardelOsin}.  This gives the bound on orders of torsion elements.

Fix $g,h\in G$ with $h$ loxodromic.  Since $G$ acts on $T$ acylindrically by hypothesis, we can apply Lemma \ref{lem:loxo-loxodromic} and Lemma~\ref{lem:elliptic-loxodromic} to obtain a uniform constant $p$ such that either $[g^p,h^p]=1$, or $g$ is elliptic and $g^p\in E(h)$, or $\langle g^p,h^p\rangle\cong \langle g^p\rangle*\langle h^p\rangle$.  If $g,h$ have infinite order, we are done (using \cite[Lemma 6.8]{Osin:acyl} to handle the case where $g$ has a power in $E(h)$).  But if $g$ has finite order, then either we conclude immediately, or, in the case where we require uniformity, we use the bound on torsion elements of vertex groups.

Hence it remains to consider the case where each of $g$ and $h$ is elliptic.  Since vertex stabilisers are closed under roots by hypothesis, $\fix{g}=\stabfix{g}$ and the same is true replacing $g$ with $h$.  Therefore, if $\stabfix{g}\cap\stabfix{h}\neq \emptyset$ then $\fix{g}\cap\fix{h}\neq \emptyset$, and applying our hypotheses to the stabiliser of any vertex fixed by both $g$ and $h$ concludes the proof.  On the other hand, if $\stabfix{g}\cap\stabfix{h}=\emptyset$, Lemma \ref{lem:elliptic-elliptic general case} shows $\langle g,h\rangle\cong F_2$.
\end{proof}

\subsubsection{Closure under roots and a generalisation}\label{subsubsec:root-closed}
Let $G$ act acylindrically on the tree $T$ without inversions.  We saw above that for any elliptic $g\in G$, the subtrees $\fix{g}\subset\stabfix{g}$ have uniformly bounded diameter.  However, this is insufficient to conclude that there is a uniform power $p$ such that $g^p$ fixes $\stabfix{g}$ pointwise, which is needed in the ``elliptic-elliptic'' case of the proof of Proposition \ref{prop:acyl-tree-action}.
The extra ``closure under roots'' property is a bit stronger than is needed to ensure this, but it is a useful condition because it is easy to verify in our examples of interest:

\begin{lem}\label{lem:closed-roots}
Let $G$ act on the tree $T$.  For $v\in\Ver T$, let $\mathcal E_v$ be the set of edges of $T$ incident to $v$ and let $\rho_v:G_v\to G_v/K(G_v)$ be the natural quotient, where $K(G_v)$ is the kernel of the $G_v$--action on $\mathcal E_v$.  Suppose that for all $v\in\Ver T$, the collection $\{\rho_v(G_e):e\in\mathcal E_v\}$ is malnormal in $G_v/K(G_v)$ and the latter group is torsion-free.  Then each vertex-stabiliser $G_w$ is closed under roots.
\end{lem}

\begin{proof}
Let $g\in G$ and suppose that $g^nw=w$ for some $n>0$ and $w\in\Ver T$.  Then there is a vertex $v$ with $gv=v$.  Let $\gamma$ be the geodesic segment from $v$ to $w$, which we can assume has length at least $1$, for otherwise we are done.  Let $e$ be the initial edge of $\gamma$.  We first show that $g\in G_e$.  
To this end, note that $g^n$ fixes $v$ and $w$, so it fixes the geodesic segment $[v,w]$ pointwise, and hence $g^n\in G_e$.  Thus, for $1\leq i\leq n-1$, $g^n\in G_{g^ie}$, so $\rho_v(g^n)\in \rho_v(g)^i\rho_v(G_e)\rho_v(g)^{-i}$ for all $i$, and the malnormality and torsion hypotheses imply $g\in K(G_v)$ and hence $g\in G_e$.  Thus $g\in G_{v'}$ where $v'$ is the terminal point of $e$, so since $d_T(v',w)<d_T(v,w)$, it follows by induction that $g\in G_w$, as required. 
\end{proof}

\subsection{Examples}\label{rem:known-examples}
Theorem \ref{thm:meta-pa} recovers various examples, some of which were known by other means.

\begin{cor}\label{cor:3-manifolds}
Let $M$ be a closed connected oriented $3$--manifold.  Then $\paf(\pi_1M)<\infty$ unless $M$ has a  Nil or Sol piece in its prime decomposition.  
\end{cor}

\begin{proof}
By the prime decomposition theorem, \cite[Theorem 0.1]{Dahmani:comb} (see also \cite[Theorem 9.2]{AFW:3man}, \cite[Corollary E]{BigdelyWise}), \cite[Lemma 3.12, Lemma 3.13]{GrovesHullLiang}, Lemma \ref{lem:closed-roots}, and Theorem \ref{thm:meta-pa}, it suffices to consider the case where either $M$ is geometric, or $M$ is a large Seifert fibred manifold with toral boundary (\emph{large} means that the base orbifold has negative Euler characteristic).  In the latter case, $\pi_1M$ is a central extension of a hyperbolic group, and hence in $\upaclass$.  In the former, examining the eight geometries, either Theorem \ref{thm:meta-pa} applies, or $M$ is a Nil or Sol manifold.  
\end{proof}

\begin{rem}\label{rem:3-man}
Most cases of Corollary \ref{cor:3-manifolds} were known by other means: we have already mentioned the 1979 result of Jaco-Shalen \cite{JacoShalen} when $M$ is, for instance, a finite-volume hyperbolic Haken manifold.  The case where $M$ is hyperbolic is covered by the (classical) Theorem \ref{thm:hyperbolic}, and, more generally, the cases where $\pi_1M$ is virtually special are covered by \cite{baudisch1981subgroups} since virtually special groups virtually embed in RAAGs \cite{HaglundWise:special}.  The latter cases include hyperbolic $3$--manifolds, graph manifolds that are either nonpositively-curved or have nonempty boundary, and mixed $3$--manifolds \citelist{\cite{Agol:virtual-haken}\cite{PrzytyckiWise:graph}\cite{Liu:graph}\cite{PrzytyckiWise:mixed}}.
\end{rem}

\begin{cor}\label{cor:free-by-Z}
Let $F$ be a finite-rank free group and let $\Phi\in \operatorname{Out}(F)$.  Then the mapping torus $G=F\rtimes_\Phi\integers$ satisfies $\paf(G)<\infty$.
\end{cor}

\begin{proof}
We will show that $G\in\upaclass$, from which the result follows by Theorem \ref{thm:meta-pa}.  By \cite[Theorem 3.5]{DahmaniLi}, $G$ is hyperbolic relative to a finite collection of subgroups, each of which is the mapping torus of a polynomially-growing automorphism, it suffices to consider the case where $\Phi$ has polynomial growth.  By \cite[Proposition 3.5]{BFH}, up to replacing $\Phi$ by $\Phi^k$, where $k>0$ depends only on $F$, we can assume that $\Phi$ is \emph{UPG} in the sense of \cite[Definition 3.10]{BFH}.  This has the effect of replacing $G$ by a bounded-index subgroup, which does not affect whether it is in $\upaclass$.  But once $\Phi$ is UPG, we can apply \cite[Proposition 2.5]{AndrewHughesKudlinska} and \cite[Lemma 5.2]{KudlinskaValiunas} to show that $G\in\upaclass$, as required.
\end{proof}

\subsection{$\upaclass$ versus hierarchical hyperbolicity}\label{subsec:questions-HHG}
The class of \emph{hierarchically hyperbolic} groups (HHGs) from \cite{BHS:II} contains hyperbolic groups, and $\integers$ central extensions of hyperbolic groups \cite{HRSS}.  It is closed under graph products \citelist{\cite{BerlyneRussell}\cite{PetytSpriano}}, and any group hyperbolic relative to a finite collection of  HHGs is again an HHG \cite{BHS:II}.  Finite graphs of groups with HHG vertex groups for which the action on the Bass-Serre tree is acylindrical are again HHGs under fairly flexible quasiconvexity conditions on the edge groups \citelist{\cite{BHS:II}\cite{BerlaiRobbio}}.  While  passing to finite-index supergroups need not preserve the property of being an HHG \cite{PetytSpriano}, the above facts, together with Theorem \ref{thm:meta-pa}, provide many hierarchically hyperbolic groups $G$ with $\paf(G)<\infty$.  On the other hand, cocompact lattices in products of locally finite trees are HHGs \cite{BHS:I}, but it is open whether these always satisfy a power alternative, and there is some evidence (see \cite{Bondarenko-squared}) against this.  

\begin{question}
Give natural conditions on a hierarchically hyperbolic group $G$ ensuring that it satisfies a (uniform) power alternative.  Specifically, provided each \emph{standard product region} (see \cite[Section 5]{BHS:II}) $P$ in $G$ has cocompact stabiliser in $G$, does the power alternative hold in $G$ provided it holds in each product region subgroup $\stab{G}{P}$?
\end{question}

We expect \cite[Corollary 14.3]{BHS:I}, \cite[Proposition 9.2]{DHS}, and \cite[Remark 2.10]{BHS:III} to be useful here.  HHGs have uniformly bounded torsion \cite{HaettelHodaPetyt} but need not be virtually torsion-free \cite{Hughes}.  In many cases, like mapping class groups, hierarchical hyperbolicity comes from an acylindrical $G$--action on a hyperbolic complex $Y$ with features reminiscent of fine hyperbolic graphs for relatively hyperbolic groups \cite{BHMS}.  

\begin{question}
Generalise Theorem \ref{thm:rel-hyp} and Propsition \ref{prop:acyl-tree-action} to prove the power alternative for groups acting acylindrically on hyperbolic simplicial complexes satisfying appropriate local conditions.
\end{question}

Such a statement would ideally be strong enough to recover a power alternative for mapping class groups, and give yet another proof for the result about graph products discussed in Remark \ref{rem:graph-prod}.  In the former case, one can imagine applying the desired generalisation to the action on the curve graph, while in the case of graph products, one might use the action on Valiunas' contact graph \cite{Valiunas:graph}.

\bibliographystyle{abbrv}
\bibliography{bibliography}
\end{document}